\documentclass[10pt]{article}

\DeclareMathAlphabet{\mathpzc}{OT1}{pzc}{m}{it}

\usepackage{amsmath}
\usepackage{amssymb}
\usepackage{amsfonts}
\usepackage{amsthm}
\usepackage{mathrsfs}
\usepackage{ifsym}
\usepackage{fontenc}
\usepackage{textcomp}
\usepackage{tipa}
\usepackage{enumerate}
\usepackage{mathtools}
\usepackage[pdftex]{color, graphicx}
\numberwithin{equation}{section}

\begin{document}
 
\title{{\bf On Ramanujan's cubic continued fraction}}       
\author{Sushmanth J. Akkarapakam and Patrick Morton}        
\date{July 26, 2024}          
\maketitle

\begin{abstract}  The periodic points of the algebraic function defined by the equation $g(x,y)=x^3(4y^2+2y+1)-y(y^2-y+1)=0$ are shown to be expressible in terms of values of Ramanujan's cubic continued fraction $c(\tau)$ with arguments in an imaginary quadratic field $K$ in which the prime $3$ splits.  If $w = (a+\sqrt{-d})/2$ lies in an order of conductor $f$ in $K$ and $9 \mid N_{K/\mathbb{Q}}(w)$, then one of these periodic points is $c(w/3)$, which is shown to generate the ring class field of conductor $2f$ over $K$.
\end{abstract}

\section{Introduction}

Ramanujan's cubic continued fraction is defined by
\begin{equation}
c(\tau)=\cfrac{q^{1/3}}{1+\cfrac{q+q^2}{1+\cfrac{q^2+q^4}{1+\cfrac{q^3+q^6}{1+\cdots}}}}.
\label{eqn:1.1}
\end{equation}
This continued fraction $c(\tau)$ can also be written in the form of an infinite product expansion, namely,
\begin{equation}
c(\tau)=q^{1/3}\prod_{n=1}^\infty\frac{(1-q^{6n-1})(1-q^{6n-5})}{(1-q^{6n-3})^2}.
\label{eqn:1.2}
\end{equation}
In \cite{se}, Selberg proved the equality of the right-hand sides of the above two equations in his first paper in $1937$, not knowing that it already appears in Ramanujan's lost notebook, but without proof. \medskip

In \cite{am}, we described the periodic points of an algebraic function related to the Ramanujan-G\"ollnitz-Gordon continued fraction $v(\tau)$. In this paper, we take a similar approach in studying the cubic continued fraction $c(\tau)$ and describe how the periodic points for a certain algebraic function arise as values of this continued fraction. \medskip

We consider negative quadratic discriminants of the form $-d \equiv 1$ (mod $3$) and arguments in the field $K = \mathbb{Q}(\sqrt{-d})$. Let $R_K$ be the ring of integers in this field and let the prime ideal factorization of (3) in $R_K$ be $(3) = \wp_3 \wp_3'$. We denote by $\Omega_f$ the ring class field over $K$ whose conductor is $f$ corresponding to the order $\textsf{R}_{-d}$ of discriminant $-d= d_K f^2$ in $K$ ($d_K$ is the discriminant of $K$). We show in Section 4 that certain values of $c(\tau)$, where  $\tau$ is the basis quotient of an ideal in $\textsf{R}_{-d}$, are periodic points of a fixed algebraic function, independent of $d$, and generate the ring class fields $\Omega_{2f}$ over $K$.\medskip

\noindent {\bf Theorem 1.1.} {\it Let $K=\mathbb{Q}(\sqrt{-d})$ with $-d=d_Kf^2\equiv1\,(\mathrm{mod}\,3)$, where $d_K=\mathrm{disc}(K/\mathbb{Q})$. Moreover, let $w \in \textsf{R}_{-d}$ be the element defined as
\begin{equation*}
w=\begin{dcases}k+\frac{\sqrt{-d}}{2},&2 \mid d\\
\frac{k+\sqrt{-d}}{2},&(2,d)=1,\end{dcases}
\end{equation*}
where $9 \mid N(w)$ and $k \equiv 1 \ (\mathrm{mod} \ 6)$. Then the value $c(w/3)$ generates the ring class field $\Omega_{2f}$ over $\mathbb{Q}$.} \bigskip

This result is a complement to the result of \cite[Thm. 13]{chkp}, according to which the value $c(\tau)$, for $\tau$ a basis quotient of an integral ideal (prime to $6$) in the maximal order $R_K$ of $K = \mathbb{Q}(\sqrt{-d})$, generates the ray class field $\Sigma_6 $ over $K$.  The field $\Sigma_6$ coincides with the ring class field $\Omega_6$ for the fields we are considering. \medskip

Recall that a periodic point $a$ of an algebraic function $y=\mathfrak{f}(z)$, with minimal polynomial $f(z,y)$ over a given field $F(z)$, is a value in the algebraic closure $\overline{F}$ for which there exist $a_1, a_2, \dots, a_{n-1} \in \overline{F}$ for which
$$f(a,a_1) = f(a_1, a_2) = \cdots = f(a_{n-2},a_{n-1}) = f(a_{n-1},a) = 0.$$
See \cite{mn}.  With this definition we prove the following. \bigskip

\noindent {\bf Theorem 1.2.} {\it Let $w \in K$ be as in Theorem 1.1.  The generator $2c(w/3)$ of the field $\Omega_{2f}$ over $\mathbb{Q}$, along with the values
$$\frac{c_1(w/6)}{c(w/3)}, \ \frac{-1}{c_1(w/6)}, \ \ \textrm{with} \ \ c_1(\tau)=c(\tau+3/2),$$
are periodic points of the algebraic function $\mathfrak{f}(z)$ defined by $f(z,\mathfrak{f}(z))=0$, where $f(x,y)=x^3(y^2+y+1)-y(y^2-2y+4)$.} \bigskip

An explicit representation of the multi-valued function $\mathfrak{f}(z)$ can be calculated using Cardan's formulas:
\begin{equation*}
\mathfrak{f}(z) = \frac{1}{3} (z^3 + 8) \left(\frac{z^3-1}{z^3+8}\right)^{1/3} + \frac{1}{3}(z^3 - 1)\left(\frac{z^3-1}{z^3+8}\right)^{-1/3} + \frac{1}{3}(z^3 + 2).
\end{equation*}

\noindent {\bf Theorem 1.3.} {\it The only periodic points of the algebraic function $\mathfrak{f}(z)$ are the fixed points $0,1,-2,\pm\sqrt{-2}$ and the conjugates over $\mathbb{Q}$ of the values listed in Theorem 1.2, as $-d$ varies over all negative quadratic discriminants satisfying $-d \equiv 1$ (mod $3$).  These periodic points are all algebraic integers.} \bigskip

This result is significant because, for one thing, the polynomial $f(x,y)$ can be used to calculate the minimal polynomials of the corresponding periodic points, by means of iterated resultants.  See Sections 4.2 and 5.  This gives an algebraic method for computing the minimal polynomials over $\mathbb{Q}$ of the values $2c(w/3)$.  Also, Theorem 1.3 shows that all the periodic points of $\mathfrak{f}(z)$ in $\overline{\mathbb{Q}}$ arise from Galois theory and the theory of complex multiplication, since the fact that $\eta = 2c(w/3)$ is periodic follows from the relation
\begin{equation}
f(\eta, \eta^{\tau_3}) = 0, \ \ \tau_3 = \left(\frac{\Omega_{2f}/K}{\wp_3}\right), \ \ K = \mathbb{Q}(\sqrt{-d}),
\label{eqn:1.3}
\end{equation}
where $\tau_3 \in \textrm{Gal}(\Omega_{2f}/K)$ is the Frobenius automorphism for the prime divisor $\wp_3 = (3,w)$ in $R_K$.  The proof of this uses the defining property of the Frobenius automorphism and the results of \cite{mc}.  The minimal period of $\eta$ with respect to $\mathfrak{f}(z)$ equals the order in $\textrm{Gal}(\Omega_{2f}/K)$ of the automorphism $\tau_3$.  This follows from the fact that considered $3$-adically, the equation  $f(x,y) = 0$ has, for a given $x \not \equiv 1$ (mod $3$) in the maximal unramified, algebraic extension $\textsf{K}_3$ of the $3$-adic field $\mathbb{Q}_3$, a unique solution in $\textsf{K}_3$:
\begin{align}
\notag y = T(x) &= \frac{1}{3}(x^3 + 2) + \frac{1}{3} (x^3 + 8) \sum_{n=0}^\infty{\binom{1/3}{n} (-1)^n \frac{3^{2n}}{(x^3+8)^n}} \\
 & + \frac{1}{3} (x^3 - 1)\sum_{n=0}^\infty{\binom{-1/3}{n} (-1)^n \frac{3^{2n}}{(x^3+8)^n}}, \ \ x \not \equiv 1 \ (\textrm{mod} \ 3).
 \label{eqn:1.4}
\end{align}
(See Lemma \ref{lem:4} in Section 5.)  With this, equation (\ref{eqn:1.3}) can be expressed in the form
\begin{equation}
\eta^{\tau_3} = T(\eta) \ \ \textrm{in} \ \textsf{K}_3, \ \ \eta = 2c(w/3),
\label{eqn:1.5}
\end{equation}
in terms of the single-valued $3$-adic function $T(x)$. This implies that $\eta= 2c(w/3)$ has period $n = \textrm{ord}(\tau_3)$ with respect to the action of $T(x)$ on $\textsf{K}_3$. \medskip

\renewcommand{\thefootnote}{\ensuremath{\fnsymbol{footnote}}}

In Section 2\footnote{Sections 2-4 and parts of Sections 5 and 6 are taken from the first author's Ph.D. dissertation.  See \cite{akk}.}, we prove modular relations for the continued fraction $c(\tau)$ which are fundamental for the whole paper.  In Section 3, we relate $c(\tau)$ to the functions $\alpha(\tau), \beta(\tau)$ from \cite{ms}, which are closely related to cubic theta functions.  Then, in Section 4, we prove the above theorems.  In Section 5, we prove several properties of the polynomials satisfied by the values $2c(w/3)$ and the iterated resultants which are used to compute them.  We also discuss the $3$-adic function $T(x)$ and prove (\ref{eqn:1.5}).  In Section 6, we compute several examples of the values $c(w/3)$.  For example, when $d = 11$, we find the value
\begin{align*}
2c\left(\frac{13+\sqrt{11}i}{6}\right) =& \left(-1 - \frac{(4i + \sqrt{11})\sqrt{3}}{9}\right)^{1/3}\\
& + \frac{(-1 + \sqrt{11}i)}{3}\left(-1 - \frac{(4i + \sqrt{11})\sqrt{3}}{9}\right)^{-1/3},
\end{align*}
which is a periodic point of $\mathfrak{f}(z)$ of period $3$, and which generates the ring class field $\Omega_2$ of conductor $2$ of the field $K = \mathbb{Q}(\sqrt{-11})$ over $\mathbb{Q}$.  We also prove the following fact.  (See Proposition \ref{prop:17}.) \medskip

\noindent {\bf Theorem 1.4.} {\it Let $w_k \in \textsf{R}_{-4^k d}$ be as in Theorem 1.1 for the discriminant $-4^k d$, where $d \equiv 7$ (mod $8$).  Then the value $2c(w_k/3)$ is a unit in the field $\Omega_{2^{k+1}f}$, for any $k \ge 0$.}
\medskip

The result of this theorem is not generally true for the values $2c(w/3)$, as the above example for $d = 11$ shows.  In that case $N_{\Omega_2/\mathbb{Q}}(2c(w/3)) = 4$.  See Example 2 in Section 6.  Finally, in Section 7, we answer the question of how to determine the correct value of $d$ in Theorems 1.1 and 1.3, if we are given the minimal polynomial of a periodic point of $\mathfrak{f}(z)$.  \medskip

We refer to \cite{b3} or \cite{bs} for Ramanujan's notation, which we use throughout Sections 2-4.  Also see the brief summary of this notation in \cite[pp. 4-7]{am}.

\section{Identities for $c(\tau)$}

\newtheorem{prop}{Proposition}

\newtheorem{thm}{Theorem}

Let us define the function $c(\tau)$ as
\begin{align*}
c(\tau)&=q^{1/3}\prod_{n=1}^\infty\frac{(1-q^{6n-1})(1-q^{6n-5})}{(1-q^{6n-3})^2}\\
&=q^{1/3}\frac{(q;q^6)_\infty(q^5;q^6)_\infty}{(q^3;q^6)_\infty^2}\\
&=q^{1/3}\frac{(q;q^2)_\infty}{(q^3;q^6)_\infty^3}\\
&=\frac{q^{1/3}\chi(-q)}{\chi^3(-q^3)}.
\end{align*}

\begin{prop}The functions $u=c(\tau)$ and $v=c(2\tau)$ satisfy the relation
$$u^2+2uv^2-v=0.$$
\label{prop:1}
\end{prop}

\begin{proof}
The functions $u=c(\tau)$ and $v=c(2\tau)$ are given by the equations
$$x=\frac{q^{1/3}\chi(-q)}{\chi^3(-q^3)}\ \ \ \text{and}\ \ \ y=\frac{q^{2/3}\chi(-q^2)}{\chi^3(-q^6)}.$$
Let $\sqrt{\alpha}$ and $\sqrt{\beta}$ be the moduli associated with the variables $q$ and $q^3$, respectively, so that
$$q=\exp{\left(-\pi\,\frac{_2F_1(\frac{1}{2},\frac{1}{2};1;1-\alpha)}{_2F_1(\frac{1}{2},\frac{1}{2};1;\alpha)}\right)}\ \ \ \text{and}\ \ \ q^3=\exp{\left(-\pi\,\frac{_2F_1(\frac{1}{2},\frac{1}{2};1;1-\beta)}{_2F_1(\frac{1}{2},\frac{1}{2};1;\beta)}\right)}.$$
See \cite{br}.  Then from \cite[p. 124, Entry 12 (vi), (vii)]{b3} (with $q=e^{-y}, x = \alpha$, respectively $q^3 = e^{-y}, x = \beta$), we can see that
$$u=2^{-1/3}\left(\frac{(1-\alpha)^2}{\alpha}\right)^{1/24}\left(\frac{\beta}{(1-\beta)^2}\right)^{1/8}\ \ \ \text{and}\ \ \ v=2^{-2/3}\left(\frac{1-\alpha}{\alpha^2}\right)^{1/24}\left(\frac{\beta^2}{1-\beta}\right)^{1/8}.$$
It follows that
$$\left(\frac{1-\alpha}{(1-\beta)^3}\right)^{1/8}=\frac{u^2}{v}\ \ \ \text{and}\ \ \ \left(\frac{\beta^3}{\alpha}\right)^{1/8}=\frac{2v^2}{u}.$$
Using the second equality in \cite[p. 230, Entry 5 (i)]{b3}, which is
$$\left(\frac{(1-\beta)^3}{1-\alpha}\right)^{1/8}-\left(\frac{\beta^3}{\alpha}\right)^{1/8}=1,$$
we obtain that
$$\frac{v}{u^2}-\frac{2v^2}{u}=1,$$
which simplifies to give us the desired relation between $c(\tau)$ and $c(2\tau)$.  This holds for $\alpha, \beta$ satisfying $0 < \alpha, \beta < 1$.  Hence, it holds for infinitely many values of $\tau$ which are bounded away from the real axis, and consequently for all $\tau$ satisfying $\mathfrak{R}(\tau) > 0$, by the principle of analytic continuation.  
\end{proof}

See \cite[p. 345, Theorem 1, (2.2)]{hc} for a very similar proof of this identity.  Following Proposition \ref{prop:9} we will give an alternate proof of this relation.

\begin{prop} The following identity holds:
$$c\left(\frac{-1}{\tau}\right)=\frac{1/2-c(\tau/2)}{1+c(\tau/2)}.$$
\label{prop:2}
\end{prop}

\begin{proof}
We have
\begin{align*}
c(\tau)&=q^{1/3}\,\frac{(q;q^6)_\infty(q^5;q^6)_\infty}{(q^3;q^6)_\infty^2}\\
&=q^{1/3}\,\frac{(q;q^2)_\infty}{(q^3;q^6)_\infty^3}\\
&=q^{1/3}\,\frac{(q;q)_\infty(q^6;q^6)_\infty^3}{(q^2;q^2)_\infty(q^3;q^3)_\infty^3}\\
&=\frac{\eta(\tau)\,\eta^3(6\tau)}{\eta(2\tau)\,\eta^3(3\tau)},
\end{align*}
where $\eta(\tau)=q^{1/24}\,(q;q)_\infty$. \medskip

Using the formula $\eta(-1/\tau)=\sqrt{-i\tau}\,\eta(\tau)$ from \cite[p. 236, Cor. 12.19]{co}, we get
\begin{align*}
c(-1/\tau)&=\frac{\eta(-1/\tau)\,\eta^3(-6/\tau)}{\eta(-2/\tau)\,\eta^3(-3/\tau)}\\
&=\frac{\sqrt{-i\tau}\,\eta(\tau)\ \big(\sqrt{-i\tau/6}\big)^3\,\eta^3(\tau/6)}{\sqrt{-i\tau/2}\,\eta(\tau/2)\ \big(\sqrt{-i\tau/3}\big)^3\,\eta^3(\tau/3)}\\
&=\frac{1}{2}\,\frac{\eta(\tau)\,\eta^3(\tau/6)}{\eta(\tau/2)\,\eta^3(\tau/3)}\\
&=\frac{1}{2}\,\frac{(q;q)_\infty(q^{1/6};q^{1/6})_\infty^3}{(q^{1/2};q^{1/2})_\infty(q^{1/3};q^{1/3})_\infty^3}.
\end{align*}
Hence, if we replace $\tau$ by $6\tau$, we obtain the following (using $\varphi(q) = (-q;q^2)_\infty^2 (q^2; q^2)_\infty$; see \cite[p. 11]{bs}):
\begin{align*}
c(-1/6\tau)&=\frac{1}{2}\cdot\frac{(q;q)_\infty^3}{(q^2;q^2)_\infty^3}\cdot\frac{(q^6;q^6)_\infty}{(q^3;q^3)_\infty}\\
&=\frac{1}{2}\cdot(q;q^2)_\infty^3\cdot\frac{1}{(q^3;q^6)_\infty}\\
&=\frac{1}{2}\cdot\frac{(q;q^2)_\infty^2(q^2;q^2)_\infty}{(q^9;q^{18})_\infty^2(q^{18};q^{18})_\infty}\cdot\frac{(q;q^2)_\infty(q^{18};q^{18})_\infty}{(q^2;q^2)_\infty(q^9;q^{18})_\infty}\cdot\frac{(q^9;q^{18})_\infty^3}{(q^3;q^6)_\infty}\\
&=\frac{1}{2}\cdot\frac{\varphi(-q)}{\varphi(-q^9)}\cdot\frac{q\,\psi(q^9)}{\psi(q)}\cdot\frac{(q^9;q^{18})_\infty^3}{q\,(q^3;q^6)_\infty}\\
&=\frac{1}{2}\cdot\big(1-2c(3\tau)\big)\cdot\left(\frac{1}{1+\frac{1}{c(3\tau)}}\right)\cdot\frac{1}{c(3\tau)}\\
&=\frac{1}{2}\cdot\frac{1-2c(3\tau)}{1+c(3\tau)}.
\end{align*}
In the penultimate line we used the results that $\displaystyle1-2c=\frac{\varphi(-q^{1/3})}{\varphi(-q^3)}$ and that $\displaystyle1+\frac{1}{c}=\frac{\psi(q^{1/3})}{q^{1/3}\,\psi(q^3)}$, both from \cite[p. 345, Entry 1 (i), (ii)]{b3}; and the definition of $c(\tau)$.  Now replacing $\tau$ by $\tau/6$ in the last equality gives us the result.
\end{proof}

An alternate proof of the above result can also be found in \cite[p. 373, Theorem 6.12]{sc}. \medskip

For later use we note that Proposition \ref{prop:2} can be expressed using the linear fractional map $A(x) = \frac{2-2x}{2+x}$, as follows:
$$2c(-1/\tau) = A(2c(\tau/2)), \ \ A(x) = \frac{2-2x}{2+x}.$$

\begin{prop} We have the following identity:
$$c^3\left(\frac{-1}{6\tau}\right)=\frac{1/8-c^3(\tau)}{1+c^3(\tau)}=t(c^3(\tau)).$$
\label{prop:3}
\end{prop}

\begin{proof}
Using the equation for $c(-1/6\tau)$ from the proof of Proposition \ref{prop:2}, we can write
\begin{align*}
c^3\left(\frac{-1}{6\tau}\right)&=\frac{1}{8}\cdot(q;q^2)_\infty^9\cdot\frac{1}{(q^3;q^6)_\infty^3}\\
&=\frac{1}{8}\cdot\frac{(q;q^2)_\infty^8(q^2;q^2)_\infty^4}{(q^3;q^6)_\infty^8(q^6;q^6)_\infty^4}\cdot\frac{(q^6;q^6)_\infty^4(q;q^2)_\infty^4}{(q^3;q^6)_\infty^4(q^2;q^2)_\infty^4}\cdot\frac{(q^3;q^6)_\infty^9}{(q;q^2)_\infty^3};
\end{align*}
so that
\begin{align*}
c^3\left(\frac{-1}{6\tau}\right)&=\frac{1}{8}\cdot\frac{\varphi^4(-q)}{\varphi^4(-q^3)}\cdot\frac{q\,\psi^4(q^3)}{\psi^4(q)}\cdot\frac{(q^3;q^6)_\infty^9}{q\,(q;q^2)_\infty^3}\\
&=\frac{1}{8}\cdot\big(1-8c^3(\tau)\big)\cdot\left(\frac{1}{1+\frac{1}{c^3(\tau)}}\right)\cdot\frac{1}{c^3(\tau)}\\
&=\frac{1}{8}\cdot\frac{1-8c^3(\tau)}{1+c^3(\tau)},
\end{align*}
where again in the penultimate line we used the facts that $\displaystyle1-8c^3=\frac{\varphi^4(-q)}{\varphi^4(-q^3)}$ and that $\displaystyle1+\frac{1}{c^3}=\frac{\psi^4(q)}{q\,\psi^4(q^3)}$, both taken from the proofs of \cite[p. 346-347, Entry 1 (i), (iii)]{b3}.
\end{proof}

\noindent {\bf Corollary.} {\it The functions $x=c(\tau)$ and $y=c(3\tau)$ satisfy the relation}
$$g(x,y)=x^3(4y^2+2y+1)-y(y^2-y+1)=0.$$

\begin{proof}
From the above two propositions, we see that
\begin{equation*}
\left(\frac{1-2c(3\tau)}{1+c(3\tau)}\right)^3=\frac{1-8c^3(\tau)}{1+c^3(\tau)},
\end{equation*}
which upon solving for $c^3(\tau)$ and rearranging gives us the result.
\end{proof}

\noindent {\bf Remark.}  $g(x,y)$ is the polynomial mentioned in the abstract.  The identity in this corollary can be re-written as
$$c^3(\tau)=c(3\tau)\,\frac{1-c(3\tau)+c^2(3\tau)}{1+2c(3\tau)+4c^2(3\tau)}.$$
This is analogous to the identity for the Rogers-Ramanujan continued fraction $r(\tau)$:
$$r^5(\tau)=r(5\tau)\,\frac{1-2r(5\tau)+4r^2(5\tau)-3r^3(5\tau)+r^4(5\tau)}{1+3r(5\tau)+4r^2(5\tau)+2r^3(5\tau)+r^4(5\tau)}.$$\bigskip

\section{The modular functions $\alpha(\tau)$ and $\beta(\tau)$}
\noindent Define the modular functions $\alpha(\tau)$ and $\beta(\tau)$ by
\begin{align}
\label{eqn:3.1} \alpha(\tau)&=3+27\frac{\eta^3(9\tau)}{\eta^3(\tau)}=3+27q\prod_{n=1}^\infty\frac{(1-q^{9n})^3}{(1-q^n)^3};\\
\label{eqn:3.2} \beta(\tau)&=3+\frac{\eta^3(\tau/3)}{\eta^3(3\tau)}=3+q^{-1/3}\prod_{n=1}^\infty\frac{(1-q^{n/3})^3}{(1-q^{3n})^3}.
\end{align}

\noindent The functions $\alpha(\tau)$ and $\beta(\tau)$ lie in the modular function fields $\textsf{K}_{\Gamma(9)}$ and $\textsf{K}_{\Gamma(3)}$, respectively, and satisfy the equation
$$\textrm{Fer}_3:\ \ 27\alpha^3+27\beta^3=\alpha^3\beta^3.$$
See \cite[Theorem 12]{ms}.

\begin{prop}
If $c=c(\tau)$ is Ramanujan's cubic continued fraction, then
\begin{align}
\alpha^3(\tau)-27&=\frac{729c^3}{(c^3+1)(8c^3-1)^2},\label{eqn:3.3}\\
\beta(\tau)&=\frac{4c^3+1}{c}.\label{eqn:3.4}
\end{align}
\label{prop:4}
\end{prop}

\begin{proof}
First observe that
\begin{align*}
\frac{4c^3+1}{c}&=\frac{3c+(4c^3-3c+1)}{c}\\
&=\frac{3c+(1-2c)^2(1+c)}{c}\\
&=3+\frac{(1-2c)^2(1+c)}{c}.
\end{align*}
Substituting
$$c=q^{1/3}\frac{(q,q^6)_\infty(q^5,q^6)_\infty}{(q^3,q^6)_\infty^2}=q^{1/3}\frac{(q,q^6)_\infty(q^5,q^6)_\infty(q^6,q^6)_\infty}{(q^3,q^6)_\infty^2(q^6,q^6)_\infty}=q^{1/3}\frac{f(-q,-q^5)}{\varphi(-q^3)}$$
and
\begin{align*}
c&=q^{1/3}\frac{(q,q^6)_\infty(q^3,q^6)_\infty(q^5,q^6)_\infty}{(q^3,q^6)_\infty^3}=q^{1/3}\frac{(q,q^2)_\infty}{(q^3,q^6)_\infty^3}=q^{1/3}\frac{(-q^3,q^3)_\infty}{(-q,q)_\infty(q^3,q^6)_\infty^2}\\
&=q^{1/3}\frac{1}{(-q,q^3)_\infty(-q^2,q^3)_\infty(q^3,q^6)_\infty^2}=q^{1/3}\frac{(q^6,q^6)_\infty}{(-q,q^3)_\infty(-q^2,q^3)_\infty(q^3,q^3)_\infty(q^3,q^6)_\infty}\\
&=q^{1/3}\frac{1}{(-q,q^3)_\infty(-q^2,q^3)_\infty(q^3,q^3)_\infty}\cdot\frac{(q^6,q^6)_\infty}{(q^3,q^6)_\infty}=q^{1/3}\frac{\psi(q^3)}{f(q,q^2)},
\end{align*}
into the above, we get
$$\frac{4c^3+1}{c}=3+\left(1-2q^{1/3}\frac{f(-q,-q^5)}{\varphi(-q^3)}\right)^2\left(1+\frac{f(q,q^2)}{q^{1/3}\psi(q^3)}\right).$$
Now from \cite[p. 49, Corollary (i), (ii)]{b3}, upon replacing $q$ by $-q^{1/3}$ and $q^{1/3}$, respectively, we obtain
\begin{align*}
\frac{4c^3+1}{c}&=3+\frac{\big(\varphi(-q^3)-2q^{1/3}f(-q,-q^5)\big)^2}{\varphi^2(-q^3)}\cdot\frac{\big(q^{1/3}\psi(q^3)+f(q,q^2)\big)}{q^{1/3}\psi(q^3)}\\
&=3+\frac{\varphi^2(-q^{1/3})}{\varphi^2(-q^3)}\cdot\frac{\psi(q^{1/3})}{q^{1/3}\psi(q^3)}.
\end{align*}
Finally, using \cite[p. 39, Entry 24 (ii)]{b3} and $f(-q) = (q; q)_\infty = q^{-1/24} \eta(\tau)$ from \cite[p.11, (1.3.15)]{bs}
\begin{align*}
\frac{4c^3+1}{c}&=3+\frac{\varphi^2(-q^{1/3})\psi(q^{1/3})}{q^{1/3}\varphi^2(-q^3)\psi(q^3)}\\
&=3+\frac{f^3(-q^{1/3})}{q^{1/3}f^3(-q^3)}\\
&=3+\frac{\eta^3(\tau/3)}{\eta^3(3\tau)}\\
&=\beta(\tau).
\end{align*}
This proves (\ref{eqn:3.4}). A similar proof can also be found in \cite[p. 348, Entry 1(iv)]{b3}. \medskip

From the equation $27\alpha^3+27\beta^3=\alpha^3\beta^3$ we obtain
\begin{align*}
\alpha^3-27&=\frac{729}{\beta^3-27}\\
&=\frac{729c^3}{(4c^3+1)^3-27c^3}\\
&=\frac{729c^3}{(4c^3-3c+1)(16c^6+12c^4+8c^3+9c^2+3c+1)}\\
&=\frac{729c^3}{(2c-1)^2(c+1)(4c^2+2c+1)^2(c^2-c+1)}\\
&=\frac{729c^3}{(c^3+1)(8c^3-1)^2},
\end{align*}
giving us (\ref{eqn:3.3}).
\end{proof}

\begin{prop}
If $c=c(\tau)$ is Ramanujan's cubic continued fraction, then we have the following relations between $c$ and the $j$-invariant $j(\tau)$:
\begin{align}
j(\tau)&=\frac{(64c^9+48c^6+228c^3+1)^3(4c^3+1)^3}{c^3(c^3+1)^3(8c^3-1)^6},\label{eqn:3.5}\\
j(3\tau)&=\frac{(64c^9+48c^6-12c^3+1)^3(4c^3+1)^3}{c^9(c^3+1)(8c^3-1)^2}.\label{eqn:3.6}
\end{align}
\label{prop:5}
\end{prop}

\begin{proof}
From \cite[p. 860, (2.7b)]{mc}, with $\mathfrak{g}(\tau) = \alpha(\tau)$, we have
\begin{equation*}
j(\tau)=\frac{\alpha^3(\alpha^3-24)^3}{\alpha^3-27}=\frac{\beta^3}{27}\left(\frac{27\beta^3}{\beta^3-27}-24\right)^3 = \frac{\beta^3(\beta^3+216)^3}{(\beta^3-27)^3}.
\end{equation*}

Now using (\ref{eqn:3.4}) to replace $\beta(\tau)$ in terms of $c(\tau)$, we get
\begin{align*}
j(\tau)&=\frac{\beta^3(\beta^3+216)^3}{(\beta^3-27)^3}\\
&=\frac{(4c^3+1)^3}{c^3}\frac{\big((4c^3+1)^3+216c^3\big)^3}{\big((4c^3+1)^3-27c^3\big)^3}\\
&=\frac{(4c^3+1)^3(4c^3+6c+1)^3(16c^6-24c^4+8c^3+36c^2-6c+1)^3}{c^3(c^3+1)^3(8c^3-1)^6}\\
&=\frac{(64c^9+48c^6+228c^3+1)^3(4c^3+1)^3}{c^3(c^3+1)^3(8c^3-1)^6},
\end{align*}
which gives us (\ref{eqn:3.5}).\medskip

From \cite[p. 860, Eq. (2.7a)]{mc}, upon replacing $z$ by $3z$ and using $\mathfrak{f}(3\tau) = \beta(\tau)$, we get
\begin{align*}
j(3\tau)&=\frac{\mathfrak{f}^3(3\tau)\big(\mathfrak{f}^3(3\tau)-24\big)^3}{\mathfrak{f}^3(3\tau)-27}\\
&=\frac{\beta^3(\beta^3-24)^3}{\beta^3-27}\\
&=\frac{(4c^3+1)^3}{c^9}\frac{\big((4c^3+1)^3-24c^3\big)^3}{\big((4c^3+1)^3-27c^3\big)}\\
&=\frac{(64c^9+48c^6-12c^3+1)^3(4c^3+1)^3}{c^9(c^3+1)(8c^3-1)^2}.
\end{align*}
This proves (\ref{eqn:3.6}).
\end{proof}

\noindent {\bf Corollary.} {\it If $\tau \in K$, an imaginary quadratic number field, then $1/c(\tau)$ is an algebraic integer.  Moreover, $2c(\tau)$ is also an algebraic integer.} \smallskip

\begin{proof}
Putting, $x = 1/c(\tau)$ in (\ref{eqn:3.5}) gives
$$j(\tau) = \frac{(x^9 + 228x^6 + 48x^3 + 64)^3 (x^3 + 4)^3}{x^6(x^3 + 1)^3(x^3 - 8)^6}.$$
The first assertion follows from the fact that the degree of the numerator in this expression is $36$, while the degree of the denominator is $33$.  Since $j(\tau)$ is an algebraic integer, $1/c(\tau)$ satisfies a monic equation with algebraic integer coefficients.  On the other hand, setting $c(\tau) = y/2$ yields
$$j(\tau) = \frac{(y^9 + 6y^6 + 228y^3 + 8)^3(y^3 + 2)^3}{y^3(y^3 + 8)^3(y^3 - 1)^6},$$
from which the second assertion follows in the same way.
\end{proof}

This is a simpler argument than the proof given in \cite[Thm. 16]{chkp} for the first assertion of the Corollary.

\subsection{More identities}
\begin{prop} The following identities hold:
\begin{align*}
\alpha\left(\frac{-1}{3\tau}\right)&=\beta(\tau),\\
\beta\left(\frac{-1}{3\tau}\right)&=\alpha(\tau),\\
\alpha^3\left(\frac{-1}{3\tau}\right)&=\frac{27\alpha^3(\tau)}{\alpha^3(\tau)-27},\\
\beta^3\left(\frac{-1}{3\tau}\right)&=\frac{27\beta^3(\tau)}{\beta^3(\tau)-27}.
\end{align*}
\label{prop:6}
\end{prop}

\begin{proof}
From the definition of $\alpha(\tau)$, we have
$$\alpha(\tau)=3+27\frac{\eta^3(9\tau)}{\eta^3(\tau)}.$$
Upon replacing $\tau$ by $-1/(3\tau)$ in the above, we find that that
\begin{align*}
\alpha\left(\frac{-1}{3\tau}\right)&=3+27\,\frac{\eta^3(-3/\tau)}{\eta^3(-1/3\tau)}\\
&=3+27\,\frac{\big(\sqrt{-i\tau/3}\big)^3\eta^3(\tau/3)}{\big(\sqrt{-3i\tau}\big)^3\eta^3(3\tau)}\\
&=3+27\,\frac{\big(\sqrt{-i\tau}\big)^3\eta^3(\tau/3)}{27\,\big(\sqrt{-i\tau}\big)^3\eta^3(3\tau)}\\
&=3+\frac{\eta^3(\tau/3)}{\eta^3(3\tau)}\\
&=\beta(\tau),
\end{align*}
where we used the transformation formula from \cite[p. 236, Cor. 12.19]{co}: $\eta(-1/\tau)=\sqrt{-i\tau}\,\eta(\tau)$.  The second equation is obtained by simply replacing $\tau$ by $-1/(3\tau)$ in the first one. And the last two are obtained by cubing the first two and using the relation between $\alpha(\tau)$ and $\beta(\tau)$.
\end{proof}

\begin{prop} The following transformation identities hold between the functions $\alpha(\tau)$ and $\beta(\tau)$:
\begin{align*}
\beta\left(\frac{-1}{\tau}\right)&=3\left(\frac{\beta(\tau)+6}{\beta(\tau)-3}\right)=\alpha(\tau/3),\\
\alpha\left(\frac{-1}{\tau}\right)&=3\left(\frac{\alpha(\tau/9)+6}{\alpha(\tau/9)-3}\right)=\beta(\tau/3).
\end{align*}
\label{prop:7}
\end{prop}
Note that the second equality in both the above equations follows from the relations in Proposition \ref{prop:6}, and we begin our proof using those relations.

\begin{proof}
From the definitions of $\alpha(\tau)$ and $\beta(\tau)$, and from Proposition \ref{prop:6},
\begin{align*}
\beta\left(\frac{-1}{\tau}\right)&=\alpha(\tau/3)\\
&=3+27\,\frac{\eta^3(3\tau)}{\eta^3(\tau/3)}\\
&=3+\frac{27}{\beta(\tau)-3}\\
&=3\left(\frac{\beta(\tau)+6}{\beta(\tau)-3}\right).
\end{align*}
A similar calculation starting with $\alpha(-1/\tau)$ gives us the second equation. 
\end{proof}

\begin{prop} The following identities hold for $c(\tau)$ and Dedekind's $\eta$-function $\eta(\tau)$:
\begin{align*}
\frac{1}{c(\tau)}-1-2\,c(\tau)&=\frac{\eta(\tau/3)\,\eta(2\tau/3)}{\eta(3\tau)\,\eta(6\tau)},\\
\frac{1}{c^3(\tau)}-7-8\,c^3(\tau)&=\frac{\eta^4(\tau)\,\eta^4(2\tau)}{\eta^4(3\tau)\,\eta^4(6\tau)}.
\end{align*}
\label{prop:8}
\end{prop}

\begin{proof}
From \cite[p. 345, Entry 1 (i), (ii)]{b3} and from the proofs of \cite[p. 346-347, Entry 1 (i), (iii)]{b3}, we have:
\begin{align*}
1-2c&=\frac{\varphi(-q^{1/3})}{\varphi(-q^3)}, \ 1+\frac{1}{c}=\frac{\psi(q^{1/3})}{q^{1/3}\,\psi(q^3)},\\
1-8c^3&=\frac{\varphi^4(-q)}{\varphi^4(-q^3)},\ \ 1+\frac{1}{c^3}=\frac{\psi^4(q)}{q\,\psi^4(q^3)}.
\end{align*}
Multiplying the first pair of relations, we get
\begin{align*}
\frac{1}{c(\tau)}-1-2\,c(\tau)&=\frac{1}{q^{1/3}}\cdot\frac{\varphi(-q^{1/3})}{\varphi(-q^3)}\cdot\frac{\psi(q^{1/3})}{\psi(q^3)}\\
&=\frac{1}{q^{1/3}}\cdot\frac{(q^{1/3};q^{2/3})_\infty^2(q^{2/3};q^{2/3})_\infty}{(q^3;q^6)_\infty^2(q^6;q^6)_\infty}\cdot\frac{(q^{2/3};q^{2/3})_\infty(q^3;q^6)_\infty}{(q^{1/3};q^{2/3})_\infty(q^6;q^6)_\infty}\\
&=\frac{1}{q^{1/3}}\cdot\frac{(q^{1/3};q^{2/3})_\infty(q^{2/3};q^{2/3})_\infty^2}{(q^3;q^6)_\infty(q^6;q^6)_\infty^2}\\
&=\frac{1}{q^{1/3}}\cdot\frac{(q^{1/3};q^{1/3})_\infty(q^{2/3};q^{2/3})_\infty}{(q^3;q^3)_\infty(q^6;q^6)_\infty}\\
&=\frac{\eta(\tau/3)\,\eta(2\tau/3)}{\eta(3\tau)\,\eta(6\tau)}.
\end{align*}
In a similar way, multiplying the second pair of relations gives us
\begin{align*}
\frac{1}{c^3(\tau)}-7-8\,c^3(\tau)&=\frac{1}{q}\cdot\frac{\varphi^4(-q)}{\varphi^4(-q^3)}\cdot\frac{\psi^4(q)}{\psi^4(q^3)}\\
&=\frac{1}{q}\cdot\frac{(q;q^2)_\infty^8(q^2;q^2)_\infty^4}{(q^3;q^6)_\infty^8(q^6;q^6)_\infty^4}\cdot\frac{(q^2;q^2)_\infty^4(q^3;q^6)_\infty^4}{(q;q^2)_\infty^4(q^6;q^6)_\infty^4}\\
&=\frac{1}{q}\cdot\frac{(q;q^2)_\infty^4(q^2;q^2)_\infty^8}{(q^3;q^6)_\infty^4(q^6;q^6)_\infty^8}\\
&=\frac{1}{q}\cdot\frac{(q;q)_\infty^4(q^2;q^2)_\infty^4}{(q^3;q^3)_\infty^4(q^6;q^6)_\infty^4}\\
&=\frac{\eta^4(\tau)\,\eta^4(2\tau)}{\eta^4(3\tau)\,\eta^4(6\tau)}.
\end{align*}
\end{proof}

See \cite{hcc} for similar proofs of the identities in Proposition \ref{prop:8}, where they are expressed in a slightly different form.  Note that these identities are analogues of similar identities for the Rogers-Ramanujan continued fraction; see \cite[Eqs. (7.2), (7.7)]{d}.

\begin{prop} The following identity holds between $\beta(-1/\tau)$ and $c=c(\tau)$:
$$\beta\left(\frac{-1}{\tau}\right)=\frac{3\,(4c^3+6c+1)}{(2c-1)^2(c+1)}=3+\frac{27c}{(2c-1)^2(c+1)}.$$
\label{prop:9}
\end{prop}

\begin{proof}
Using (\ref{eqn:3.4}) and Proposition \ref{prop:7}, we see that
\begin{align*}
\beta\left(\frac{-1}{\tau}\right)&=3\left(\frac{\beta(\tau)+6}{\beta(\tau)-3}\right)=3\left(\frac{\frac{4c^3+1}{c}+6}{\frac{4c^3+1}{c}-3}\right)\\
&=\frac{3\,(4c^3+6c+1)}{(4c^3-3c+1)}=\frac{3\,(4c^3+6c+1)}{(2c-1)^2(c+1)}\\
&=3+\frac{27c}{(2c-1)^2(c+1)}.
\end{align*}
\end{proof}

\noindent {\it Alternate proof of Proposition 1.}\medskip

Here we give the argument that Proposition \ref{prop:1} can also be derived from Propositions \ref{prop:2}, \ref{prop:4} and \ref{prop:9}.  From (\ref{eqn:3.4}), we have
$$\beta\left(\frac{-1}{\tau}\right)=\frac{4c^3\left(\frac{-1}{\tau}\right)+1}{c\left(\frac{-1}{\tau}\right)}.$$
Let $x=c(\tau/2)$ and $y=c(\tau)$.  Using the result of Proposition \ref{prop:9} for the left-hand side, and that of Proposition \ref{prop:2} for the right-hand side, we obtain
$$\frac{3\,(4y^3+6y+1)}{(2y-1)^2(y+1)}=\frac{4\left(\frac{1/2-x}{1+x}\right)^3+1}{\left(\frac{1/2-x}{1+x}\right)}.$$
This simplifies to
$$\frac{3\,(4y^3+6y+1)}{(2y-1)^2(y+1)}=\frac{3\,(2x^3-6x^2-1)}{(2x-1)\,(x+1)^2}.$$
Moving everything to the left hand side and factoring the numerator, we see that
$$\frac{27\,(2xy+1)\,(x^2+2xy^2-y)}{(2x-1)\,(x+1)^2(2y-1)^2(y+1)}=0.$$
Now from the definitions, it is clear that $y=O(q^{1/3})$ and $x=O(q^{1/6})$ as $q$ tends to $0$. So it is the last factor that vanishes  for $q$ sufficiently small, and hence vanishes for $|q|<1$ by the identity theorem. Therefore,
$$x^2+2xy^2-y=0.$$
Now replacing $\tau$ by $2\tau$ gives the result. \qed \bigskip

We now describe the series of calculations and numerical approximations that led to the result in the next Proposition \ref{prop:10}. Note that from Proposition \ref{prop:5}, we have the following relationship between $j(\tau)$ and $c=c(\tau)$:
$$j(\tau)=\frac{(64c^9+48c^6+228c^3+1)^3(4c^3+1)^3}{c^3(c^3+1)^3(8c^3-1)^6}.$$
Now, replacing $c^3$ by $x$, let
$$j_3(x)=\frac{(64x^3+48x^2+228x+1)^3(4x+1)^3}{x(x+1)^3(8x-1)^6}.$$
Then
$$j(\tau)=j_3(c^3(\tau)).$$
Hence,
$$j\left(\frac{-1}{6\tau}\right)=j_3\left(c^3\left(\frac{-1}{6\tau}\right)\right)=j_3(t(c^3(\tau))),$$
by the identity $c^3(-1/6\tau) = t(c^3(\tau))$, where $t(x)$ is given by Proposition \ref{prop:3} as $t(x)=\frac{1/8-x}{1+x}$.\bigskip

\noindent A calculation on Maple shows that
$$j_{33}(x)=j_3(t(x))=-\frac{(2x-1)^3(8x^3+12x^2+6x-1)^3}{x^6(8x-1)(x+1)^2}.$$
Therefore,
\begin{equation}
j\left(\frac{-1}{6\tau}\right)=j_{33}(c^3(\tau)).
\label{eqn:3.7}
\end{equation}

From \cite[p. 860, (2.7b)]{mc} the function
\begin{equation*}
\alpha(\tau)=3+27\,\frac{\eta^3(9\tau)}{\eta^3(\tau)}
\end{equation*}
satisfies the following relation:
\begin{equation}
j(\tau)=\frac{\alpha^3(\alpha^3-24)^3}{\alpha^3-27}=G(\alpha^3(\tau))
\label{eqn:3.8},
\end{equation}
where
$$G(x)=\frac{x(x-24)^3}{x-27}.$$
If we make the substitution $x\rightarrow(x^{-2/3}-2x^{1/3})^3=(12-8x-6/x+1/x^2)$, we find that
$$G\big((x^{-2/3}-2x^{1/3})^3\big)=-\frac{(2x-1)^3(8x^3+12x^2+6x-1)^3}{x^6(8x-1)(x+1)^2}=j_{33}(x).$$
It follows from $(\ref{eqn:3.7})$ and $(\ref{eqn:3.8})$ that
\begin{align*}
G\left(\left(\frac{1}{c^2(\tau)}-2c(\tau)\right)^3\right)&=G\big(12-8c^3(\tau)-6c^{-3}(\tau)+c^{-6}(\tau)\big)\\
&=j\left(\frac{-1}{6\tau}\right)=G\left(\alpha^3\left(\frac{-1}{6\tau}\right)\right).
\end{align*}

\noindent This raises the question: what is the relationship between $\alpha(-1/6\tau) = \beta(2\tau)$ and $\displaystyle\frac{1}{c^2(\tau)}-2c(\tau)$?  The answer is given in the next proposition.

\begin{prop} The following relation holds between $\beta(2\tau)$ and $c(\tau)$:
$$\beta(2\tau)=\frac{1}{c^2(\tau)}-2\,c(\tau).$$
\label{prop:10}
\end{prop}

\begin{proof}
We know the relation between $x=c(\tau)$ and $y=c(2\tau)$ from Proposition \ref{prop:1} as
$$x^2+2xy^2=y.$$
Multiply both sides by $2xy$ to get 
$$2x^3y+4x^2y^3=2xy^2.$$
We add the above two equations to obtain
$$2x^3y+4x^2y^3+x^2=y.$$
Now separate the variables in the following way:
\begin{align*}
4x^2y^3+x^2&=y-2x^3y\\
x^2(4y^3+1)&=y(1-2x^3)\\
\frac{4y^3+1}{y}&=\frac{1-2x^3}{x^2},
\end{align*}
which, upon using (\ref{eqn:3.4}) for the left hand side, gives us the result.
\end{proof}

\begin{prop}
The functions $x=\beta(\tau)$ and $y = \beta(2\tau)$ satisfy the relation
\begin{equation*}
h(x,y) = x^3+y^3 -x^2y^2+9xy-54 = 0,
\end{equation*}
as do the functions $x=\alpha(\tau)$ and $y=\alpha(2\tau)$.
\label{prop:10a}
\end{prop}

\begin{proof}
Propositions \ref{prop:4} and \ref{prop:10} give that
\begin{align*}
\beta(\tau) &= \frac{(2c)^3+2}{2c},\\
\beta(2\tau) & = \frac{-(2c)^3+4}{(2c)^2}, \ \ c = c(\tau).
\end{align*}
On the other hand, the curve $h(x,y) = 0$ has genus $0$ and the parametrization
$$x = \frac{t^3+2}{t}, \ y = \frac{-t^3+4}{t^2}.$$
Setting $t = 2c(\tau)$ proves the first assertion.  For the second, we note that
$$(x-3)^3(y-3)^3 h\left(\frac{3(x+6)}{x-3},\frac{3(y+6)}{y-3}\right) = 3^9h(x,y).$$
Putting $x = \beta(\tau), y=\beta(2\tau)$ and using Proposition \ref{prop:7} yields that
$$(\beta(\tau)-3)^3 (\beta(2\tau)-3)^3 h(\alpha(\tau/3),\alpha(2\tau/3)) = 3^9 h(\beta(\tau),\beta(2\tau)) = 0.$$
Hence, $h(\alpha(\tau/3),\alpha(2\tau/3)) = 0$, and replacing $\tau$ by $3\tau$ yields the assertion.
\end{proof}

\section{Periodic points}

In this section we will prove Theorems \ref{thm:1}-\ref{thm:3} of the Introduction.

\subsection{Modular function invariance}

We first express the cubic continued fraction $c(\tau)$ in terms of the theta constants.
\begin{align*}
c(\tau)&=q^{1/3}\prod_{n=1}^\infty\frac{(1-q^{6n-1})(1-q^{6n-5})}{(1-q^{6n-3})^2}\\
&=q^{1/3}\prod_{n=1}^\infty\frac{(1-q^{6n})(1-q^{6n-1})(1-q^{6n-5})}{(1-q^{6n})(1-q^{6n-3})(1-q^{6n-3})}\\
&=q^{1/3}\prod_{n=1}^\infty\frac{[1-(q^6)^n][1-(q^6)^{n-\frac{1}{6}}][1-(q^6)^{n-\frac{5}{6}}]}{[1-(q^6)^n][1-(q^6)^{n-\frac{3}{6}}][1-(q^6)^{n-\frac{3}{6}}]}\\
&=q^{1/3}\prod_{n=1}^\infty\frac{[1-(q^6)^n][1-(q^6)^{n-\frac{1+2/3}{2}}][1-(q^6)^{n-\frac{1-2/3}{2}}]}{[1-(q^6)^n][1-(q^6)^{n-\frac{1+0}{2}}][1-(q^6)^{n-\frac{1-0}{2}}]}.
\end{align*}
The above can be re-written as
\begin{equation*}
c(\tau)=e^{-2\pi i/6}\cdot\frac{e^{2\pi i(\frac{2/3\cdot1}{4})}(q^6)^{\frac{(2/3)^2}{8}}}{e^{2\pi i(\frac{0\cdot1}{4})}(q^6)^{\frac{0^2}{8}}}\cdot\prod_{n=1}^\infty\frac{[1-(q^6)^n][1-(q^6)^{n-\frac{1+2/3}{2}}][1-(q^6)^{n-\frac{1-2/3}{2}}]}{[1-(q^6)^n][1-(q^6)^{n-\frac{1+0}{2}}][1-(q^6)^{n-\frac{1-0}{2}}]},
\end{equation*}
which, upon using the product formula for theta constants in \cite[p. 143, (4.8)]{d}, gives us that
\begin{equation}
c(\tau)=e^{-2\pi i/6}\cdot\frac{\theta[\genfrac{}{}{0pt}{1}{2/3}{1}](6\tau)}{\theta[\genfrac{}{}{0pt}{1}{0}{1}](6\tau)}.
\label{eqn:4.1}
\end{equation}

Let $\gamma=\bigl(\begin{smallmatrix}a&b\\c&d\end{smallmatrix}\bigr)\in\text{SL}_2(\mathbb{Z})$ so that $\gamma\tau=\frac{a\tau+b}{c\tau+d}$. Then, from Lemma A.1 in \cite[p. 158, Eq. (A.2)]{d}, it follows that if $\gamma \in \Gamma_1(N)$, then
\begin{align*}
\theta\left[\genfrac{}{}{0pt}{0}{\epsilon}{\epsilon'}\right](N\gamma\tau)&=\theta\left[\genfrac{}{}{0pt}{0}{\epsilon}{\epsilon'}\right]\textstyle\big(N(\frac{a\tau+b}{c\tau+d})\big)\\
&=\lambda\sqrt{c\tau-d}\cdot e^{2\pi i\big(\frac{-\epsilon^2Nb(a-2)+2\epsilon\epsilon'(bc+d-1)}{8}\big)}\cdot\theta\left[\genfrac{}{}{0pt}{0}{\epsilon}{\epsilon'}\right](N\tau),
\end{align*}
if $N$ and $N\epsilon$ are both even, where $\lambda$ is some eighth root of unity that does not depend on $\epsilon$ or $\tau$.  Hence,
\begin{align*}
c(\gamma\tau)&=e^{-2\pi i/6}\cdot\frac{\theta[\genfrac{}{}{0pt}{1}{2/3}{1}](6\gamma\tau)}{\theta[\genfrac{}{}{0pt}{1}{0}{1}](6\gamma\tau)}\\
&=e^{-2\pi i/6}\cdot\frac{\lambda\sqrt{c\tau-d}\cdot e^{2\pi i\big(\frac{-\frac{4}{9}\cdot6b(a-2)+2\cdot\frac{2}{3}\cdot1(bc+d-1)}{8}\big)}\theta[\genfrac{}{}{0pt}{1}{2/3}{1}](6\tau)}{\lambda\sqrt{c\tau-d}\cdot e^{2\pi i\big(\frac{0+0}{8}\big)}\theta\left[\genfrac{}{}{0pt}{1}{0}{1}\right](6\tau)}\\
&=e^{-2\pi i/6}\cdot\frac{e^{2\pi i(-\frac{b(a-2)}{3}+\frac{bc+d-1}{6})}}{1}\cdot\frac{\theta[\genfrac{}{}{0pt}{1}{2/3}{1}](6\tau)}{\theta[\genfrac{}{}{0pt}{1}{0}{1}](6\tau)}.
\end{align*}
Therefore,
\begin{equation*}
c(\gamma\tau)=e^{2\pi i\big(-\frac{b(a-2)}{3}+\frac{bc+d-1}{6}\big)}\cdot c(\tau),
\end{equation*}
from which we conclude that $c(\tau)$ is invariant under $\Gamma(6)$, as well as that $c^3(\tau)$ is invariant under $\Gamma_1(6)$.

\subsection{Minimal polynomials of periodic points}
As in \cite{mc}, we let $p_d(x)$ denote the minimal polynomial of $\beta(w/3)$, where $w$ satisfies the conditions of Theorem 1.1. ($\beta(w/3)$ is the same as $\alpha=\mathfrak{f}(w)$ from \cite[Theorem 1]{mc}.)  Then $c(w/3)$ is a root of the polynomial $C_d(x)=x^{2h(-d)}p_d\big(\frac{x^3+2}{x}\big)$, by equation (\ref{eqn:3.4}).  We will see that the roots of these factors are periodic points of the function $\mathfrak{f}(z)$ mentioned in the Introduction.  To find the minimal polynomials of these periodic points, we proceed as follows.  \medskip

We know that the pair $(x,y) = (c(\tau), c(3\tau))$ satisfies the equation $g(x,y) = 0$, where
$$g(x,y) = x^3(4y^2 + 2y + 1) - y(y^2 - y + 1).$$
Then $8g(x/2, y/2) = f(x,y)$, where
$$f(x,y) = x^3 (y^2 + y + 1) - y(y^2 - 2y + 4).$$
We have
$$f(2c(\tau),2c(3\tau)) = 0.$$
Proceeding as in \cite{am}, we compute the minimal polynomials of periodic points using iterated resultants involving the polynomial $f(x,y)$.  We set
$$R^{(1)}(x,x_1) = f(x, x_1) = x^3(x_1^2+ x_1+1)-x_1(x_1^2-2x_1+4)$$
and define, inductively,
$$R^{(n)}(x,x_n) = \textrm{Res}_{x_{n-1}}(R^{(n-1)}(x,x_{n-1}),f(x_{n-1},x_n)) \ \ n \ge 2.$$
Then the roots of the polynomial
$$R_n(x) = R^{(n)}(x,x), \ \ n \ge 1,$$
are the periodic points whose minimal periods divide $n$. Also, let
$$P_n(x) = \prod_{d \mid n}{R_d(x)^{\mu(n/d)}}.$$
It will be shown that this equation for $P_n(x)$ describes a polynomial by proving that $R_d(x)$ divides $R_n(x)$ if $d$ divides $n$ and that $R_n(x)$ has distinct roots.  See Proposition \ref{prop:14} in Section 5. \medskip

Table \ref{tab:2} lists the irreducible factors $q(x)$ of $P_n(x)$ for $1\le n\le4$ along with their discriminants, and the integer $d$ for which $q(x)$ divides $C_d(x)=x^{2h(-d)}p_d\big(\frac{x^3+2}{x}\big)$ in the right-most column.\medskip

\begin{table}
\caption{Polynomials dividing $P_n(x)$ for the identity between $2c(\tau)$ and $2c(3\tau)$}
\bigskip
\begin{center}
\begin{tabular}{|c|c|c|c|}
\hline
$n$ & Factors of $P_n(x)$ & discriminant & $d$\\
\hline
$1$ & $x, x - 1, x + 2$ & $1$ & \\
$1$ & $x^2 + 2$ & $-2^3$ & 8\\
$2$ & $x^4 - 4x^3 + 4x^2 + 8$ & $2^{14}3^2$ & 32\\
$2$ &  $x^4 - 2x^3 + 2x^2 + 4x + 4$ & $2^8 3^2 5^2$ &20\\
$2$ &  $x^4 + 2x^2 + 4x + 2$ & $2^8 3^2$ & 8\\
$3$ &  $x^6 + 2x^4 + 4x^3 + 12x^2 + 4x + 4$ & $-2^{12} 3^6 11^3$ & 11\\
$3$ &  $x^6 + 5x^4 - 3x^3 + 12x^2 + 4x + 8$  & $-2^6 3^6 5^2 23^3$ & 23\\
$3$ & $x^6 - 10x^5 + 44x^4 - 56x^3 + 16x^2 - 32x + 64$ &  $-2^{30} 3^6 23^3$ & 92(6)\\
$3$  & $x^6 - x^5 + 6x^4 + 3x^3 + 10x^2 + 8$ & $-2^6 3^6 5^2 23^3$ & 23\\
$3$  &  $x^6 - 2x^5 + 12x^4 - 8x^3 + 8x^2 + 16$ & $-2^{22} 3^6 11^3$ & 44(6)\\
$3$ &  $x^6 + x^5 + x^4 + 7x^3 + 11x^2 + 5x + 1$ & $-3^6 23^3$ & 23\\
$3$ & $x^{12} - 12x^{11} + 44x^{10} - 16x^9 - 88x^8 + 176x^7+ 48x^6$ & $2^{98} 3^{30} 13^6$ & 104(6)\\
&  $\ \ \ - 352x^5 - 352x^4 + 128x^3 + 704x^2 + 384x + 64$ & $-$ & \\
$4$ & $x^8 + 4x^7 + 8x^6 + 52x^4 + 32x^2 - 32x + 16$ & $2^{44} 3^{12} 7^4$ & 56\\
$4$  & $x^8 + 4x^7 + 4x^6 + 8x^5 + 34x^4 + 24x^3 + 12x^2 - 8x + 2$ & $2^{22} 3^{12} 5^4 23^2$ & 32\\
$4$  & $x^8 + 8x^7 + 24x^6 - 96x^5 + 272x^4 - 128x^3 $ & $2^{64} 3^{12} 5^4 23^2$ & 128(8)\\
&$+ 128x^2 - 256x + 128$&$-$ & \\
$4$ & $x^8 + 24x^6 - 32x^5 + 88x^4 - 64x^3 + 64x^2 - 64x + 64$ & $2^{48} 3^{12} 5^4 11^2$ & 80\\
$4$ & $x^8 + 2x^7 + 4x^6 + 8x^5 + 22x^4 + 16x^3 + 24x^2 + 4$ & $2^{20} 3^{12} 5^4 11^2$ & 20\\
$4$ & $x^8 + 6x^7 + 6x^6 - 4x^5 + 72x^4 + 8x^3 + 24x^2 - 48x + 16$ & $2^{42} 3^{12} 17^4$ & 68(8)\\
$4$ & $x^{16} - 48x^{15} + 752x^{14} - 3136x^{13} + 4768x^{12} $ & $2^{232} 3^{56} 5^{16} 11^4$ & 320(8)\\
&  $- 3072x^{11} + 3328x^{10}- 16896x^9 + 47488x^8  $ & $\cdot 17^4 53^4 71^2$ & \\
& $- 43008x^7+ 96256x^6 - 133120x^5+ 45056x^4 $ & $-$ & \\
&$- 98304x^3 + 90112x^2 + 16384$&$-$ & \\
$4$ & $x^{16} - 4x^{15} + 4x^{14} + 32x^{13} - 28x^{12} + 8x^{11} + 160x^{10}$ & $2^{96} 3^{56} 7^8$ & 56\\
& $+ 100x^8 - 240x^7 + 512x^6 + 1344x^5 $ & $\cdot 11^4 29^4 47^2$ & \\
&$+ 2288x^4 + 1664x^3 + 704x^2 + 16$&$-$ & \\
$4$ & $x^{16} + 22x^{14} + 48x^{13} + 44x^{12} + 260x^{11} + 376x^{10} $ & $2^{52} 3^{56} 5^{16} 11^4$ & 80\\
&  $+ 336x^9 + 742x^8+ 528x^7 + 208x^6 + 384x^5 $ & $\cdot 17^4 53^4 71^2$ & \\
&$+ 1192x^4 + 1568x^3 + 752x^2 + 96x + 4$&$-$ & \\\hline
\end{tabular}
\end{center}
\label{tab:2}
\end{table}
\medskip

There are three additional factors of $P_4(x)$ of degree $16$ not listed in Table 1, given by
\begin{align*}
x^{16} &+ 176x^{14} - 832x^{13} + 2288x^{12} - 2688x^{11} + 2048x^{10} + 1920x^9\\
& \ \ + 1600x^8 + 10240x^6 - 1024x^5 - 7168x^4 - 16384x^3 + 4096x^2\\
& \ \ + 8192x + 4096,\\
x^{16}& - 8x^{15} + 288x^{14} - 576x^{13} - 968x^{12} + 5152x^{11} - 1536x^{10} \\
& \ \  - 9408x^9 + 19504x^8 + 18816x^7 - 6144x^6 - 41216x^5 - 15488x^4\\
& \ \  + 18432x^3 + 18432x^2 + 1024x + 256,
\end{align*}
which are factors of $C_{224}(x)$ and $C_{260}(x)$, respectively, and
\begin{align*}
x^{16}& - 36x^{15} + 560x^{14} - 1344x^{13} - 2252x^{12} + 9920x^{11} - 4736x^{10}\\
& \ \  - 19872x^9 + 35200x^8 + 39744x^7 - 18944x^6 - 79360x^5 - 36032x^4 \\
& \ \  + 43008x^3+ 35840x^2 + 4608x + 256,
\end{align*}
which is a factor of $C_{308}(x)$.  Their discriminants are given by $2^{216} 3^{56} 7^8 11^4 29^4 47^2$, $2^{208} 3^{56} 5^8 7^8 13^8 47^4$, and $2^{184} 3^{56} 5^8 7^{16} 11^8 23^8$, respectively.  \medskip

In Table 1, an integer $(n)$ in parentheses in the last column indicates the period of the cofactor of the given factor of $C_d(x)$. \medskip
 
We would like to know what field a root of each of these polynomials generates. Observations from numerical calculations and Table 1 give rise to the following conjecture.\medskip
 
\noindent {\bf Conjecture.} (a) {\it Let $p_d(x)$ be the minimal polynomial of $\beta(w/3)$, where
$$w = k + \frac{\sqrt{-d}}{2} \ \ \textrm{or} \ \ \frac{k+\sqrt{-d}}{2}, \ \ k \equiv 1 \ (\textrm{mod} \ 6),$$
according as $-d \equiv 1$ (mod $3$) is even or odd, and where $9 \mid N(w)$.  Then the irreducible factors of the polynomial
\begin{equation*}
C_d(x) = x^{2h(-d)}p_d\left(\frac{x^3+2}{x}\right)
\end{equation*}
are minimal polynomials of periodic points of the algebraic function $\mathfrak{f}(z)$ defined by $f(z,\mathfrak{f}(z))=0$.}\smallskip
 
\noindent (b) {\it If $q(x)$ is an irreducible factor of $C_d(x)$, then the odd prime divisors of $\textrm{disc}(q(x))$ also divide $\textrm{disc}(p_d(x))$.}
\smallskip
 
\noindent (c) {\it $C_d(x)$ factors over $\mathbb{Q}$ into: 
\begin{enumerate}[(i)]
\item 3 irreducibles of degree $2h(-d)$; 
\item a product of two irreducible factors of degrees $2h(-d)$ and $4h(-d)$; 
\item or is irreducible,
\end{enumerate}
according as $\displaystyle \left(\frac{-d}{2}\right) = 1, 0$ or $-1$.  In all cases a root of an irreducible factor of largest degree generates the ring class field $\Omega_{2f}$ over $K = \mathbb{Q}(\sqrt{-d})$.}
\medskip

We will prove this conjecture in the subsequent section.  See the Remark after the proof of Theorem \ref{thm:1}.

\subsection{Roots of a cubic}
Our goal in this subsection is to find all three roots of the cubic equation $4x^3-\beta(\tau)x+1=0$, which is the relation satisfied by $c(\tau)$ and $\beta(\tau)$. It is already clear that one root of the equation is $c(\tau)$, from (\ref{eqn:3.4}).\bigskip

We have
$$c(\tau)=q^{1/3}\,\frac{(q;q^6)_\infty(q^5;q^6)_\infty}{(q^3;q^6)_\infty^2},$$
where $q=e^{2\pi i\tau}$. Let us define
\begin{align*}
c_1(\tau)&=c(\tau+3/2)\\
&=-q^{1/3}\,\frac{(-q;q^6)_\infty(-q^5;q^6)_\infty}{(-q^3;q^6)_\infty^2}.
\end{align*}

\begin{prop}(a) The following identity holds for $c(\tau)$ and $c_1(\tau)$:
$$c(\tau)+c_1(\tau)+2c^2(\tau)c_1^2(\tau)=0;$$
(b) The following relation holds for $c(\tau)$ and $c_1(\tau/2)$:
$$2c^2(\tau)c_1(\tau/2)+c_1^2(\tau/2)-c(\tau)=0.$$
\label{prop:11}
\end{prop}

\begin{proof}
(a) First observe that
\begin{align}
c(\tau)\cdot c_1(\tau)&=q^{1/3}\,\frac{(q;q^6)_\infty(q^5;q^6)_\infty}{(q^3;q^6)_\infty^2}\cdot-q^{1/3}\,\frac{(-q;q^6)_\infty(-q^5;q^6)_\infty}{(-q^3;q^6)_\infty^2}\nonumber\\
&=-q^{2/3}\,\frac{\big[(q;q^6)_\infty(-q;q^6)_\infty\big]\cdot\big[(q^5;q^6)_\infty(-q^5;q^6)_\infty\big]}{(q^3;q^6)_\infty^2(-q^3;q^6)_\infty^2}\nonumber\\
&=-q^{2/3}\,\frac{(q^2;q^{12})_\infty\cdot(q^{10};q^{12})_\infty}{(q^6;q^{12})_\infty^2}\nonumber\\
&=-c(2\tau)\label{eqn:4.2}.
\end{align}
From Proposition \ref{prop:1}, we know that
$$c^2(\tau)+2c(\tau)c^2(2\tau)-c(2\tau)=0.$$
Replacing $c(2\tau)$ by $-c(\tau)c_1(\tau)$ in the above, we get
$$c^2(\tau)+2c^3(\tau)c_1^2(\tau)+c(\tau)c_1(\tau)=0,$$
giving us the desired identity after cancelling a factor of $c(\tau)$ from both sides. See \cite[p. 345, Theorem 1, (2.1)]{hc} for a similar proof of this identity. \medskip

\noindent (b) To prove the relation between $c(\tau)$ and $c_1(\tau/2)$, we employ both Proposition \ref{prop:1} as well as the identity from (a).\medskip

\noindent Using (\ref{eqn:4.2}) in Proposition \ref{prop:11}(a), we get
$$c(\tau)+c_1(\tau)+2c^2(2\tau)=0.$$
After replacing by $\tau$ by $\tau/2$, this can be written as
\begin{equation}
c(\tau/2)+2c^2(\tau)=-c_1(\tau/2)
\label{eqn:4.3}.
\end{equation}
Similarly, upon replacing $\tau$ by $\tau/2$ in the result of Proposition \ref{prop:1}, we find that
$$c^2(\tau/2)+2c^2(\tau)c(\tau/2)-c(\tau)=0,$$
$$c(\tau/2)\cdot\big[c(\tau/2)+2c^2(\tau)\big]-c(\tau)=0.$$
Now, substituting (\ref{eqn:4.3}) into the above, we see that
$$\big[-2c^2(\tau)-c_1(\tau/2)\big]\cdot\big[-c_1(\tau/2)\big]-c(\tau)=0,$$
which gives us
$$2c^2(\tau)c_1(\tau/2)+c_1^2(\tau/2)-c(\tau)=0,$$
finishing the proof.
\end{proof}

This next proposition spells out what the three roots of the cubic are.

\begin{prop} The three roots of the cubic equation $4x^3-\beta(\tau)x+1=0$ are explicitly given by $c(\tau), \frac{c_1(\tau/2)}{2c(\tau)}$ and $\frac{-1}{2c_1(\tau/2)}$.
\label{prop:12}
\end{prop}

\begin{proof}
We prove this in a straightforward way by considering combinations of the roots, in order to find the polynomial that has these three expressions as its roots. From (\ref{eqn:3.4}), we already know that $c(\tau)$ is one root of the above cubic.\medskip

\noindent First, note that the product of the above three roots is clearly $-1/4$. The sum of the roots is
$$c(\tau)+\frac{c_1(\tau/2)}{2c(\tau)}-\frac{1}{2c_1(\tau/2)}=\frac{2c^2(\tau)c_1(\tau/2)+c_1^2(\tau/2)-c(\tau)}{2c(\tau)c_1(\tau/2)},$$
which is $0$ from the result of Proposition \ref{prop:11}(b).\medskip

\noindent Finally, consider the sum of the products of the roots taken two at a time. This sum is
\begin{align*}
\frac{c_1(\tau/2)}{2}-\frac{1}{4c(\tau)}-\frac{c(\tau)}{2c_1(\tau/2)}&=\frac{2c(\tau)c_1^2(\tau/2)-c_1(\tau/2)-2c^2(\tau)}{4c(\tau)c_1(\tau/2)}\\
&=\frac{2c(\tau)\big[c_1^2(\tau/2)-c(\tau)\big]-c_1(\tau/2)}{4c(\tau)c_1(\tau/2)}.
\end{align*}
From Proposition \ref{prop:11}(b), substitute $c_1^2(\tau/2)-c(\tau)=-2c^2(\tau)c_1(\tau/2)$ into the last equation. We find that
\begin{align*}
\frac{2c(\tau)\cdot\big[-2c^2(\tau)c_1(\tau/2)\big]-c_1(\tau/2)}{4c(\tau)c_1(\tau/2)}&=\frac{-4c^3(\tau)c_1(\tau/2)-c_1(\tau/2)}{4c(\tau)c_1(\tau/2)}\\
&=-\frac{4c^3(\tau)+1}{4c(\tau)}\\
&=-\frac{\beta(\tau)}{4},
\end{align*}
by (\ref{eqn:3.4}). From these combinations of the roots, we see that the cubic polynomial with the given three roots has to be $4x^3-\beta(\tau)x+1$.
\end{proof}

\noindent {\bf Remark.} Knowing the three roots of this cubic will allow us to say exactly what the periodic points of the function $f(x,y)$ are, as values of the cubic continued fraction $c(\tau)$, after we prove the main results in the next section. Namely, they are all the conjugates of the values $2c(w/3)$, of $c_1(w/6)/c(w/3) = -1/c(w/6)$, and of $-1/c_1(w/6)$, for the appropriately chosen $w$, where $c_1(\tau)=c(\tau+3/2)$.

\subsection{Class fields generated by values of $c(\tau)$}
In this subsection, we will determine the class fields over $K = \mathbb{Q}(\sqrt{-d})$ generated by suitable values of $c(\tau)$. \medskip

Let $\mathcal{O}$ denote the order of conductor $f$ in the imaginary quadratic field $K$, and let $\Omega_f$ denote the ring class field of conductor $f$ for $K$ corresponding to $\mathcal{O}$.  Also, let $d_K$ denote the discriminant of $K$.  The following well-known degree formula can be found in \cite[p. 132, Theorem 7.24]{co}:
$$[\Omega_f:K]=h(\mathcal{O})=\frac{h(\mathcal{O}_K)f}{[\mathcal{O}_K^*:\mathcal{O}^*]}\prod_{p\mid f}\left(1-\left(\frac{d_K}{p}\right)\frac{1}{p}\right),$$
where $\left(\frac{d_K}{p}\right)$ is the Kronecker symbol.  We assume the discriminant $-d$ of the order $\mathcal{O} = \textsf{R}_{-d}$ satisfies
$$-d=f^2d_K \equiv 1 \ (\textrm{mod} \ 3).$$
The relative degrees of different ring class fields can be calculated using the above formula.  We find that
\begin{align*}
[\Omega_{6f}:\Omega_{f}]&=6\left(1-\left(\frac{d_K}{2}\right)\frac{1}{2}\right)\left(1-\left(\frac{d_K}{3}\right)\frac{1}{3}\right)\\
&=4\left(1-\left(\frac{d_K}{2}\right)\frac{1}{2}\right),
\end{align*}
if the conductor $f$ is odd.  On the other hand, if $f$ is even, the factor corresponding to the prime $p = 2$ drops out, since this factor is present for both $6f$ and for $f$.  In this case we have $[\Omega_{6f}:\Omega_{f}] =4$.  Putting this together with the above formula for odd $f$, we find that
\begin{equation}
[\Omega_{6f}:\Omega_{f}]=\begin{cases}
4, &2 \mid d_K f^2 = -d, \\
2, &-d \equiv 1 \ (\text{mod} \ 8),\\
6, &-d \equiv5 \ (\text{mod} \ 8).
\end{cases}
\label{eqn:4.4}
\end{equation}
In the same way, we can also see that
$$[\Omega_{2f}:\Omega_{f}]=\begin{cases}
2, & 2 \mid d,\\
1, &-d \equiv 1 \ (\text{mod} \ 8),\\
3, &-d \equiv 5 \ (\text{mod} \ 8);
\end{cases}$$
and that
$$[\Omega_{3f}:\Omega_{f}]=2.$$
Hence, we find that
$$[\Omega_{6f}:\Omega_{2f}]=2.$$

We next prove two useful lemmas before we get to the main theorem of this section.

\newtheorem{lem}{Lemma}

\begin{lem} Let $w \in \mathcal{O}$ be the element in \cite[Thm. 1]{mc}:
$$w = \begin{dcases}
k+\frac{\sqrt{-d}}{2}, & \textrm{if} \  2 \mid d,\\
\frac{k+\sqrt{-d}}{2}, & \textrm{if} \ (2,d) = 1,
\end{dcases}$$
where $k^2 \equiv -d/4$, respectively $-d$ (mod $9$) and $k \equiv 1$ (mod $6$).  Then the quantity
$$\beta\left(\frac{2w}{3}\right) = 3 + \left(\frac{\eta(2w/9)}{\eta(2w)}\right)^3$$
generates $\Omega_{2f}$ over $\mathbb{Q}$.
\label{lem:1}
\end{lem}

\begin{proof}
If $d$ is odd, then
$$2w = 2\left(\frac{k+\sqrt{-d}}{2}\right) = k + \sqrt{-d},$$
where $2w$ and $k$ satisfy the conditions of \cite[(1.1)]{mc} for $4d$.  Hence, $\beta(2w/3)$ (denoted $\alpha$ in \cite{mc}) generates $\Omega_{2f}$, by Theorem 1 of \cite{mc}.  Now assume $d$ is even.  Then $2w = 2k + \sqrt{-d}$.  To put this in the required form we add $9$.  Then $2w+9 = 2k+9 +\sqrt{-d}$, where $k' = 2k+9$ is odd.  Then
\begin{align*}
\left(\frac{\eta((2w+9)/9)}{\eta(2w+9)}\right)^3 &= \left(\frac{\eta(\frac{2w}{9}+1)}{\eta(2w+9)}\right)^3\\
& = \frac{\zeta_8\eta(\frac{2w}{9})^3}{\zeta_8^9 \eta(2w)^3}\\
& = \frac{\eta(\frac{2w}{9})^3}{\eta(2w)^3}.
\end{align*}
Noting that $k' \equiv -1$ (mod $3$), $\eta(-k'+\sqrt{-d})$ is the complex conjugate of $\eta(k'+\sqrt{-d})$.  
Then $w' = -k'+\sqrt{-d}$ has the required form, and
$$\beta(w'/3) = \overline{\beta((2w+9)/3)} = \overline{\beta(2w/3)}$$
generates $\Omega_{2f}$, by \cite[Thm. 1]{mc}.  Since $\Omega_{2f}$ is normal over $\mathbb{Q}$, this shows $\beta(2w/3)$ generates $\Omega_{2f}$ also when $d$ is even.
\end{proof}

\begin{lem} The ideal $\wp_3$ is not ramified in the field $\mathbb{Q}(c(w/3))$ over $K=\mathbb{Q}(\sqrt{-d})$.
\label{lem:2}
\end{lem}
\begin{proof}
Let $w$ be as in Lemma \ref{lem:1}, and let $(3)=\wp_3\wp_3'$ and $\wp_3=(3,w)$ in the ring of integers $R_K$ of the field $K=\mathbb{Q}(\sqrt{-d})$. From the results of \cite[p. 862, Lemma 2.3(b)]{mc}, we know that the principal ideal $(\beta(w/3)-3) = \wp_3'^3$, so that $\wp_3^3$ does not divide $(\beta(w/3)-3)$. This quantity $(\beta(w/3)-3)$ is referred to as $\gamma^3$ in that paper.\medskip

We know that $c(w/3)$ is a root of the cubic polynomial given by $4x^3-\beta(w/3)x+1=0$, whose discriminant is 
$$16(\beta^3(w/3)-27) = 16[\beta(w/3)-3][\beta^2(w/3)+3\beta(w/3)+9].$$
This shows that $\wp_3$ does not divide the discriminant of the polynomial satisfied by $c(w/3)$ over $\Omega_f$, so its prime divisors don't ramify in $\Omega_f(c(w/3))$ over $\Omega_f$. Since $f$ is not divisible by $3$, it does not ramify in $\Omega_f$ over $K$ either. Hence it does not ramify in $\Omega_f(c(w/3))$ over $K$.
\end{proof}

The next theorem shows that values of the continued fraction $c(\tau)$, for appropriately chosen arguments, are generators for the ring class field of conductor $2f$ over $\mathbb{Q}$.

\begin{thm}
Let $K=\mathbb{Q}(\sqrt{-d})$ with $-d=d_Kf^2\equiv1\,(\mathrm{mod}\,3)$, where $d_K=\mathrm{disc}(K/\mathbb{Q})$. Moreover, let $w$ be the element defined as
\begin{equation*}
w=\begin{dcases}k+\frac{\sqrt{-d}}{2},&2\,|\,d\\\frac{k+\sqrt{-d}}{2},&(2,d)=1,\end{dcases}
\end{equation*}
where $9\mid N(w)$ and $k\equiv1\,(\mathrm{mod}\,6)$. Then the value $c(w/3)$ generates $\Omega_{2f}$ over $\mathbb{Q}$.
\label{thm:1}
\end{thm}

\begin{proof}
We start by showing that $c(w/3)$ lies in $\Omega_{2f}$, and then we use the relation between $\beta(2\tau)$ and $c(\tau)$ to show that $c(w/3)$ generates $\Omega_{2f}$. First, it easily follows from a theorem of Schertz in \cite[p. 128, Theorem 5.2.1]{sch} with $N=6$, that $c(w/3)$ lies in $\Omega_{6f}$, because $c(\tau)$ is a modular function for $\Gamma(6)$ and $c(-1/\tau)$ has rational $q$-coefficients (Prop. \ref{prop:2}). We would like to bring this down to $\Omega_{2f}$. If we let $L=\Omega_f(c(w/3))$ so that $L\subseteq\Omega_{6f}$, we know that $L$ is an abelian extension of $K$ whose conductor $\mathfrak{f}$ divides $6f$ in $K$. We also know that in between the fields $\Omega_{f}$ and $\Omega_{6f}$ lie the fields $\Omega_{2f}$ and $\Omega_{3f}$. From the discussion preceding Lemma \ref{lem:1}, we have three possibilities for the index $[\Omega_{6f}:\Omega_{f}]$, depending on the value of $d$ modulo $8$. We consider the three cases separately.
\medskip

First, consider the case where $-d \equiv 5$ (mod $8$) and $[\Omega_{6f}:\Omega_f]=6$. Since the Galois group $\textrm{Gal}(\Omega_{6f}/\Omega_{f})$ is abelian, it has to be the cyclic group of order $6$, in which case there is exactly one proper subgroup for each index and these correspond to the fields $\Omega_{2f}$ and $\Omega_{3f}$. We can see that $L$ cannot be $\Omega_{3f}$ because we know $\wp_3$ ramifies in $\Omega_{3f}$ over $K$, but not in $L/K$, by Lemma $2$. So $c(w/3)$ cannot generate $\Omega_{3f}$. Therefore, $L$ cannot be $\Omega_{3f}$ or $\Omega_{6f}$, and the only possibility is $L \subseteq \Omega_{2f}$.  \medskip

Now let's consider the case where $-d \equiv 1$ (mod $8$) and $[\Omega_{6f}:\Omega_f]=2$. Here the Galois group is the $2$-element group. Prime divisors of $3$ ramify in $\Omega_{6f}$ over $K$, but not in $\Omega_f$ over $K$. Therefore, $c(w/3)$ has to lie in $\Omega_f$, which in this case is the same as $\Omega_{2f}$.  \medskip

The third and the trickiest case is when $d$ is even and $[\Omega_{6f}:\Omega_f]=4$. In this case the Galois group $G$ is the Klein four-group, because $\Omega_{2f}$ and $\Omega_{3f}$ are both quadratic over $K$. $G$ has three subgroups of index $2$, and two of these correspond to the fields $\Omega_{2f}$ and $\Omega_{3f}$. They are different quadratic extensions, as the prime divisors of $3$ ramify in one but not the other. If we let $\Omega_{2f}=\Omega_f(\sqrt{a})$ and $\Omega_{3f}=\Omega_f(\sqrt{b})$, then the third field is $F=\Omega_f(\sqrt{ab})=\Omega_f(\pm\sqrt{a}\sqrt{b})$. Both fields $\Omega_{2f}$ and $\Omega_{3f}$ are normal over $\mathbb{Q}$. We claim that $F$ is also normal over $\mathbb{Q}$. If $\varphi$ is an isomorphism taking $F/\mathbb{Q}$ to another field $F'/\mathbb{Q}$, then $\varphi$ extends to an automorphism of $\Omega_{6f}$ and has to take $\Omega_{2f}$ and $\Omega_{3f}$ to themselves, as they are each normal over $\mathbb{Q}$.  But then $\varphi$ takes $F$ to a subfield of $\Omega_{6f}$ which contains $\Omega_f$ and which must be $F$ itself, since $\varphi$ is $1-1$ on these subfields.  Hence, $F$ is normal over $\mathbb{Q}$. This implies that if $L = \Omega_f(c(w/3))$ is equal to $F=\Omega_f(\sqrt{ab})$, then $\wp_3$ would not ramify in $F$, so $\wp_3'$ would not ramify either.  Hence, $3$ would not be ramified in $F/\mathbb{Q}$. However, we claim that the prime divisors of $3$ do ramify in this field. We know from \cite[p. 76, Theorem 31]{dm} that if a prime is not ramified in two field extensions, then it is not ramified in their composite field. If $\wp_3$ or $\wp_3'$ doesn't ramify in $\Omega_f(\sqrt{ab})$ over $K$, then since it doesn't ramify in $\Omega_{2f} = \Omega_f(\sqrt{a})$, it would not ramify in $\Omega_f(\sqrt{a}, \sqrt{b})$, which is $\Omega_{6f}$.  So $\wp_3$ must ramify in $\Omega_f(\sqrt{ab})$, hence $L$ is not equal to this field. Therefore, we conclude again that that $c(w/3)$ has to lie in $\Omega_{2f}$.  \medskip

From Lemma \ref{lem:1}, we already know that $\beta(2w/3)$ generates $\Omega_{2f}$, and we have from Proposition \ref{prop:10} that
$$\beta(2w/3)=\frac{1}{c^2(w/3)}-2\,c(w/3).$$
Since $\beta(2w/3)$ is a rational function in $c(w/3)$, the fact that $c(w/3)$ lies in $\Omega_{2f}$ establishes that $c(w/3)$ must generate $\Omega_{2f}$ over $\mathbb{Q}$, as claimed.
\end{proof}

\noindent {\bf Corollary.} {\it If $w$ is defined as in Theorem \ref{thm:1}, with $-d = d_K f^2$, the value
$$\xi=\frac{\eta(w/9)\,\eta(2w/9)}{\eta(w)\,\eta(2w)}$$
lies in the ring class field $\Omega_{2f}$ of $K = \mathbb{Q}(\sqrt{-d})$.} \medskip

\begin{proof} Set $\tau = w/3$ in the first identity of Proposition \ref{prop:8}.  Then the fact that $\frac{1}{c(w/3)}-1-2c(w/3)$ lies in $\Omega_{2f}$ implies the assertion.
\end{proof}

\noindent {\bf Remark.} The polynomial $C_d(x)$ from the above conjecture, given by
\begin{equation}
C_d(x) = x^{2h(-d)}p_d\left(\frac{x^3+2}{x}\right),
\label{eqn:4.5}
\end{equation}
can be written as the norm (product of the conjugates) of the cubic polynomial $x^3-\beta(w/3)x+2$ for $\Omega_f/\mathbb{Q}$, where $p_d(x)$ is a product of $(x-\beta)$ over all the roots $\beta$ and has degree $2h(-d)$. Since a root of $x^3-\beta(w/3)x+2$ has degree $1, 2$ or $3$ over $\mathbb{Q}(\beta(w/3)) = \Omega_f$ and the roots $\beta$ are all conjugate to each other over $\mathbb{Q}$, it is clear that the degrees of the factors of $C_d(x)$ are either $2h(-d)$, $4h(-d)$ or $6h(-d)$.  Hence, we see that $C_d(x)$ either factors over $\mathbb{Q}$ into: (i) 3 irreducibles of degree $2h(-d)$; (ii) a product of two irreducible factors of degrees $2h(-d)$ and $4h(-d)$; or is (iii) irreducible, according as $\displaystyle \left(\frac{-d}{2}\right) = 1, 0$ or $-1$.  This establishes the Conjecture, Part c) in Section 4.2.  The result of Theorem \ref{thm:1} says that in all cases a root of an irreducible factor of the largest degree generates the ring class field $\Omega_{2f}$ over $K = \mathbb{Q}(\sqrt{-d})$. \medskip
 
Theorem \ref{thm:1} is a complement to the result of \cite[Thm. 13]{chkp}, which says that the value $c(\tau)$, where $\tau$ is the basis quotient for an ideal in the maximal order $R_K$ whose norm is relatively prime to $6$, generates the ray class field $\Sigma_6$ over $K = \mathbb{Q}(\sqrt{-d})$.  For the fields we are considering, $\Sigma_6 = \Omega_6$ for $-d = d_K$; while the value $c(w/3)$ in Theorem \ref{thm:1} generates $\Omega_{2} = \Sigma_2$ over $\mathbb{Q}$, when $f=1$.  \medskip

The following theorem is the main result of this paper, showing that suitable values of the Ramanujan's cubic continued fraction are periodic points of a fixed algebraic function.

\begin{thm}
Let $K=\mathbb{Q}(\sqrt{-d})$ with $-d=d_Kf^2\equiv1\,(\mathrm{mod}\,3)$, where $d_K=\mathrm{disc}(K/\mathbb{Q})$ and let $w$ be the element defined as
\begin{equation*}
w=\begin{dcases}k+\frac{\sqrt{-d}}{2},&2\,|\,d\\\frac{k+\sqrt{-d}}{2},&(2,d)=1,\end{dcases}
\end{equation*}
where $9\mid N(w)$ and $k\equiv1\,(\mathrm{mod}\,6)$. Then the generator $2c(w/3)$ of the field $\Omega_{2f}$ over $\mathbb{Q}$ and all the roots of $C_d(x)$ in $(\ref{eqn:4.5})$ are periodic points of the algebraic function $\mathfrak{f}(z)$ defined by $f(z,\mathfrak{f}(z))=0$, where $f(x,y)=x^3(y^2+y+1)-y(y^2-2y+4)$.
\label{thm:2}
\end{thm}

\begin{proof}
The identity between $\beta(\tau)$ and $c(\tau)$ in (\ref{eqn:3.4}) can be written as
$$\beta(\tau)=\frac{\big(2c(\tau)\big)^3+2}{\big(2c(\tau)\big)}.$$
Let $g(x,y) = x^3(4y^2+2y+1)-y(y^2-y+1)$, which is the relation satisfied by $x=c(\tau)$ and $y=c(3\tau)$ in the Corollary to Proposition \ref{prop:3}. Then let
\begin{equation}
f(x,y):=8g(x/2,y/2)=x^3(y^2+y+1)-y(y^2-2y+4),
\label{eqn:4.6}
\end{equation}
so that $f(2c(\tau),2c(3\tau))=0$.
\medskip

Let $p_d(x)$ be the minimal polynomial of $\beta(w/3)$ over $\mathbb{Q}$. From \cite[p. 856, Theorem 2]{mc}, the roots of $p_d(x)$ are the periodic points of the algebraic function whose minimal polynomial over $\mathbb{Q}(x)$ is 
$$g_1(x,y) = (y^2 + 3y + 9)x^3 - (y + 6)^3.$$
Then we claim that the irreducible factors of the polynomial
$$C_d(x) = x^{2h(-d)}p_d\left(\frac{x^3+2}{x}\right)$$
are minimal polynomials of periodic points of the algebraic function $\mathfrak{f}(z)$ defined by $f(z,\mathfrak{f}(z)) = 0$. This will follow from the identity
\begin{align*}
x^3 y^3 \cdot& g_1\left(\frac{x^3 + 2}{x}, \frac{y^3 + 2}{y} \right) = (x^3y^2 + x^3y + x^3 - y^3 + 2y^2 - 4y)\\
& \times (x^6 y^5 - x^6 y^4 + 3x^6 y^3 + x^3 y^6 + 2x^6 y^2 + 2x^3 y^5 + 4x^6 y + 6x^3 y^4 \\
& \ \ \ + 2x^3 y^3 + 28x^3 y^2 + 24 x^3 y - 8y^4 - 16y^3 - 24y^2 - 16y - 8),
\end{align*}
where the first factor on the right is equal to $f(x,y)$. Note that the degree of $p_d(x)$ is $2h(-d)$ from \cite{mc}.
\medskip

Let $-d \equiv 1$ (mod $3$).  Let $\eta$ be any root of the polynomial $x^3-\beta(w/3)x+2=0$. Then $\eta$ lies in either $\Omega_f$ or $\Omega_{2f}$. We show that $\eta$ is a periodic point of the function $\mathfrak{f}(z)$ satisfying $f(z, \mathfrak{f}(z)) = 0$, by showing that
$$f(\eta,\eta^{\tau_3}) = 0, \ \ \textrm{with} \ \eta = 2c(w/3), \ \ \tau_3 = \left(\frac{\Omega_{2f}/K}{\wp_3}\right).$$
Recall that $\tau_3$ is the unique automorphism in $\textrm{Gal}(\Omega_{2f}/K)$ for which
$$\alpha^{\tau_3} \equiv \alpha^3 \ (\textrm{mod} \ \wp_3),$$
for all $\alpha \in \Omega_{2f}$ which are integral (or integral for $\wp_3$).  Now
$$\tau_3|_{\Omega_f} =  \left(\frac{\Omega_{f}/K}{\wp_3}\right) = \tau,$$
so that with $b=\beta(w/3) \in \Omega_f$, we have $b^{\tau_3} = b^{\tau}$.  Since $g_1(b,b^\tau) = 0$, by \cite[p. 873, (4.12)]{mc} (noting that $b$ is denoted by $\alpha$ in that paper), the above identity, using $x = \eta, y = \eta^{\tau_3}$ and $b = \frac{\eta^3+2}{\eta}$, gives that
$$f(\eta,\eta^{\tau_3}) k(\eta,\eta^{\tau_3}) = 0,$$
where
\begin{align*}
k(x,y) & = x^6 y^5 - x^6 y^4 + 3x^6 y^3 + x^3 y^6 + 2x^6 y^2 + 2x^3 y^5 + 4x^6 y + 6x^3 y^4 \\
& \ \ \ + 2x^3 y^3 + 28x^3 y^2 + 24 x^3 y - 8y^4 - 16y^3 - 24y^2 - 16y - 8.
\end{align*}
Reducing this polynomial mod $3$, we find that
$$k(x,y)\equiv(x^6y +x^3y^2+1)(y+2)^4\,(\textrm{mod}\,3)$$
from which it follows that
$$k(x,x^3)\equiv2(x+2)^{21}\,(\textrm{mod}\,3).$$
We want to show that $f(\eta,\eta^{\tau_3}) = 0$.  We show $k(\eta,\eta^{\tau_3}) \neq 0$.  Considering $k(\eta,\eta^{\tau_3})$ mod $\wp_3$ yields that
$$k(\eta,\eta^{\tau_3}) \equiv k(\eta,\eta^3) \equiv 2(\eta + 2)^{21} \ (\textrm{mod} \ \wp_3).$$
If $\eta \equiv 1$ (mod $\mathfrak{p}$), for some prime divisor $\mathfrak{p}$ of $\wp_3$ in $\Omega_{2f}$, then the relation $\eta^3-b \eta+2 = 0$ implies that $b \equiv 0$ (mod $\mathfrak{p}$).  But $\mathfrak{p} \mid \wp_3$, and according to \cite[p.862, Lemma 2.3]{mc} $b =\beta(w/3)$ (called $\alpha$ in that paper) is relatively prime to $\wp_3$.  It follows that $k(\eta,\eta^{\tau_3}) \not \equiv 0$ mod $\wp_3$, which implies that $k(\eta,\eta^{\tau_3}) \neq 0$.  This proves that $f(\eta,\eta^{\tau_3}) = 0$. \medskip

If the order of $\tau_3$ is $n$, then applying powers of $\tau_3$ gives us that
$$f(\eta,\eta^{\tau_3})=f(\eta^{\tau_3},\eta^{\tau_3^2})=\dots=f(\eta^{\tau_3^{n-1}},\eta)=0,$$
which implies that $\eta$ is a periodic point of $\mathfrak{f}(z)$.
\end{proof}

This establishes the Conjecture, Part a) in Section 4.2.  Noting the Remark after Proposition \ref{prop:12}, this also implies the assertion of Theorem 1.2 in the Introduction.  \medskip

We also note the relation
\begin{equation}
x^3 y^3 f\left(\frac{-2}{x},\frac{-2}{y}\right) = 8f(x,y);
\label{eqn:4.7}
\end{equation}
this implies that the map $x \rightarrow -2/x$ maps the set of periodic points of period $n$ to itself.  Furthermore, it is easy to check that the map $A(x) = \frac{2-2x}{2+x}$ derived from Proposition \ref{prop:2} has the following effect:
\begin{equation*}
(2+x)^3(2+y)^3 f(A(x),A(y)) = 6^3 f(y,x).
\end{equation*}
There is a similar formula involving $B(x) = A(-2/x) = \frac{x+2}{x-1}$:
\begin{equation}
(x-1)^3(y-1)^3 f(B(x),B(y)) = 3^3 f(y,x).
\label{eqn:4.8}
\end{equation}

\begin{thm}
The only periodic points of the algebraic function $\mathfrak{f}(z)$ are the fixed points $0,1,-2,\pm\sqrt{-2}$ and all the roots of $C_d(x)$ in (\ref{eqn:4.5}), as $-d$ varies over all negative quadratic discriminants satisfying $-d \equiv 1$ (mod $3$).
\label{thm:3}
\end{thm}

\begin{proof}
Let the polynomials $g_1(x,y), f(x,y)$ and $k(x,y)$ be the same as in the proof of Theorem \ref{thm:2}, so that we have the identity
$$x^3y^3\cdot g_1\left(\frac{x^3+2}{x},\frac{y^3+2}{y}\right)=f(x,y)\cdot k(x,y).$$
From \cite[Theorem 2]{mc}, we know that the only periodic points of the algebraic function $y=\mathfrak{g}(z)$ whose minimal polynomial over $\mathbb{Q}(x)$ is $g_1(x,y)$ are its fixed point $x=3$ and the values $\beta(w/3)$ and its conjugates, as $-d$ ranges over discriminants $\equiv 1$ (mod $3$) of imaginary quadratic orders.\medskip

First, we find the fixed points by setting $y=x$ in $f(x,y)$ in $(\ref{eqn:4.6})$. This gives us
$$f(x,x)=x^5+x^4+2x^2-4x=x(x-1)(x+2)(x^2+2).$$
Hence, the values $0,1,-2,\pm\sqrt{-2}$ are the fixed points of $\mathfrak{f}(z)$.\medskip

Now let $\eta$ be a periodic point different from the fixed points of the algebraic function $\mathfrak{f}(z)$ such that $f(z,\mathfrak{f}(z))=0$. Then there exist $\eta_1=\eta,\eta_2,\eta_3,\dots,\eta_n\in\overline{\mathbb{Q}}$ such that 
$$f(\eta_1,\eta_2)=f(\eta_2,\eta_3)=\dots=f(\eta_n,\eta_1)=0.$$
If we let $\lambda_i=\frac{\eta_i^3+2}{\eta_i}$, then the above identity will give us that
$$g_1(\lambda_1,\lambda_2)=g_1(\lambda_2,\lambda_3)=\dots=g_1(\lambda_n,\lambda_1)=0,$$
indicating that $\lambda_1$ is a periodic point of function $\mathfrak{g}(z)$. Note that $\lambda_i \neq 3$, since $\eta_i \neq 1, -2$, which are fixed points that we are avoiding. Hence, by the results of \cite[Theorem 2]{mc}, this shows that $\lambda_1$ is a conjugate of $\beta(w/3)$ for some discriminant $-d$ and hence a root of the same minimal polynomial $p_d(x)$. Since $\lambda_1=\frac{\eta^3+2}{\eta}$, this shows that $\eta$ is a root of $C_d(x)$ in (\ref{eqn:4.5}). This completes the proof.
\end{proof}

Since all periodic points of $\mathfrak{f}(z)$ with periods greater than $1$ are roots of some polynomial $C_d(x)$, they must all be algebraic integers.  Equation (\ref{eqn:4.7}) shows that the map $U(x) = \frac{-2}{x}$ maps periodic points to periodic points, and therefore $\frac{2}{\alpha}$ is an algebraic integer, for any periodic point $\alpha$.  We can now prove Part b) of the Conjecture in Section 4.2.

\begin{prop}
If $q(x)$ is an irreducible factor of $C_d(x)$, then any odd prime factor of $\textrm{disc}(q(x))$ also divides $\textrm{disc}(p_d(x))$.  If the roots of $q(x)$ are units, the same assertion holds for the prime $2$.
\label{prop:18}
\end{prop}

\begin{proof}
If the odd prime $l$ divides $\textrm{disc}(q(x))$, let $\mathfrak{l}$ be a prime ideal divisor of $l$ in $\Omega_{2f}$ over $K = \mathbb{Q}(\sqrt{-d})$ and $\gamma_1$ and $\gamma_2$ be distinct roots of $q(x)$ for which $\mathfrak{l} \mid (\gamma_1-\gamma_2)$.  Also, let $\beta_i = \frac{\gamma_i^3+2}{\gamma_i}$ be the corresponding roots of $p_d(x)$.  Considering the equation
\begin{align*}
\beta_1-\beta_2 & = \frac{\gamma_1^3+2}{\gamma_1}-\frac{\gamma_2^3+2}{\gamma_2}\\
& = \frac{\gamma_1^3 \gamma_2-\gamma_2^3 \gamma_1+2(\gamma_2-\gamma_1)}{\gamma_1 \gamma_2}\\
& = (\gamma_1-\gamma_2)\left((\gamma_1+\gamma_2)-\frac{2}{\gamma_1 \gamma_2}\right)
\end{align*}
shows that $\mathfrak{l} \mid (\beta_1-\beta_2)$, since the term $\displaystyle \frac{2}{\gamma_1 \gamma_2} = \frac{1}{2} \frac{2}{\gamma_1} \frac{2}{\gamma_2}$ is integral for $\mathfrak{l}$.  Thus $\mathfrak{l}$ divides $\textrm{disc}(p_d(x))$, which implies that $l \mid \textrm{disc}(p_d(x))$.  If $l = 2$ and the roots $\gamma_1$ and $\gamma_2$ are units, then the same assertion holds, since in that case $\displaystyle \frac{2}{\gamma_1 \gamma_2} = \frac{2}{\gamma_1} \frac{1}{\gamma_2}$ is an algebraic integer.
\end{proof}

\section{A $3$-adic function.}

In this section we show that the expressions $P_n(x)$ defined in Section 4.2 are polynomials.  The discussion is similar to the discussion in \cite[Section 3]{mn}.

\begin{prop} We have the congruences
\begin{align*}
R^{(n)}(x,x_n) & \equiv (-1)^{n-1}(x^{3^n} -x_n)(x_n -1)^{3^n-1} \ (\textrm{mod} \ 3);\\
R_n(x) & \equiv (-1)^{n-1}(x^{3^n} -x)(x -1)^{3^n-1} \ (\textrm{mod} \ 3).
\end{align*}
\label{prop:13}
\end{prop}

\begin{proof}
For $n = 1$ we have, with $f(x,y)=x^3(y^2+y+1)-y(y^2-2y+4)$, that
\begin{align*}
R^{(1)}(x,x_1)=f(x,x_1)&=x^3(x_1^2+x_1+1)-x_1(x_1^2-2x_1+4)\\
 & \equiv x^3(x_1^2-2x_1+1)-x_1(x_1^2-2x_1+1)\ (\textrm{mod}\ 3)\\
 & \equiv (x^3-x_1)\,(x_1-1)^2\ (\textrm{mod}\ 3).
\end{align*}
Hence,
$$R_1(x) \equiv (x^3-x)\,(x-1)^2\ (\textrm{mod}\ 3).$$
Now for the induction step, assume the result is true for $n-1$. Then,
\begin{align*}
&R^{(n)}(x,x_n) = \textrm{Res}_{x_{n-1}}(R^{(n-1)}(x,x_{n-1}),f(x_{n-1},x_n))\\
&\equiv\textrm{Res}_{x_{n-1}}\big((-1)^n(x^{3^{n-1}} -x_{n-1})(x_{n-1} -1)^{3^{n-1}-1},(x_{n-1}^3-x_n)(x_n-1)^2\big)\ (\textrm{mod}\ 3).
\end{align*}
Recall that the resultant of two polynomials $\displaystyle f=\sum_{i=0}^n{a_ix^i}$ and $\displaystyle g=\sum_{i=0}^m{b_ix^i}$, having roots $\alpha_1,\alpha_2,\ldots,\alpha_n$ and $\beta_1,\beta_2,\ldots,\beta_m$, respectively, is defined as
$$\textrm{Res}(f,g)=a_n^m\prod_{i=1}^n{g(\alpha_i)},$$
and
$$\textrm{Res}(g,f)=(-1)^{mn}\textrm{Res}(f,g).$$
The roots of $(x_{n-1}^3-x_n)\,(x_n-1)^2$, as a polynomial in $x_{n-1}$, are $\omega^i\sqrt[3]{x_n}$ for $i=0,1,2$, where $\omega=e^{2\pi i/3}$. Hence,
\begin{align*}
\textrm{Res}_{x_{n-1}}&((-1)^n(x^{3^{n-1}} -x_{n-1})(x_{n-1} - 1)^{3^{n-1}-1},(x_{n-1}^3-x_n)(x_n-1)^2)\\
&=(-1)^{3^{n-1}\cdot3}\big(x_n-1\big)^{2\cdot3^{n-1}} (-1)^n\big(x^{3^{n-1}}-\sqrt[3]{x_n}\big)\big(\sqrt[3]{x_n}-1\big)^{3^{n-1}-1}\\
& \ \ \ \times(-1)^n\big(x^{3^{n-1}}-\omega\sqrt[3]{x_n}\big) \big(\omega\sqrt[3]{x_n}-1\big)^{3^{n-1}-1}\\
& \ \ \ \times(-1)^n\big(x^{3^{n-1}}-\omega^2\sqrt[3]{x_n}\big) \big(\omega^2\sqrt[3]{x_n}-1\big)^{3^{n-1}-1}\\
&=(-1)^{3^{n}}(-1)^{3n}\big(x_n-1\big)^{2\cdot3^{n-1}} \big(x^{3^n}-x_n\big) \big(x_n-1\big)^{3^{n-1}-1}\\
&=(-1)^{n-1}\big(x^{3^n}-x_n\big) \big(x_n-1\big)^{3^n-1}.
\end{align*}
Hence, we obtain
\begin{align*}
R^{(n)}(x,x_n) & \equiv (-1)^{n-1}(x^{3^n} -x_n)(x_n -1)^{3^n-1} \ (\textrm{mod} \ 3),\\
R_n(x) & \equiv (-1)^{n-1}(x^{3^n} -x)(x -1)^{3^n-1} \ (\textrm{mod} \ 3),
\end{align*}
completing the induction.
\end{proof}

\begin{lem} The polynomial $R^{(n)}(x,x_n)$ has the form
$$R^{(n)}(x,x_n) = (-1)^{n-1}(x_n^2+x_n+1)(x_n-1)^{3^n-3}x^{3^n}+S_n(x,x_n),$$
where
\begin{align*}
\textrm{deg}_x \ S_n(x,x_n) &\le 3^n-3,\\
\textrm{deg}_{x_n} \ S_n(x,x_n) &= 3^n.
\end{align*}
\label{lem:3}
\end{lem}
\begin{proof}
We prove this by induction.  The statement is clear for $n=1$, since $R^{(1)}(x,x_1) = f(x,x_1) = x^3(x_1^2+x_1+1)-(x_1^3-2x_1^2+4x_1)$, with $S_1(x,x_1) = -(x_1^3-2x_1^2+4x_1)$.  Proceeding by induction from $n-1$ to $n$, we have
\begin{align*}
R^{(n)}(x,x_n) &= \textrm{Res}_{x_{n-1}}(R^{(n-1)}(x,x_{n-1}),f(x_{n-1},x_n)) \\
& = -\textrm{Res}_{x_{n-1}}(f(x_{n-1},x_n), R^{(n-1)}(x,x_{n-1}))\\
& = -(x_n^2+x_n+1)^{3^{n-1}} \prod_i{R^{(n-1)}(x,\rho_i)},
\end{align*}
where $\rho_i, \ 1 \le i \le 3$, are the roots of $f(x_{n-1},x_n) = 0$, as a polynomial in $x_{n-1}$.
Thus, for $n \ge 2$,
\begin{align*}
R^{(n)}&(x,x_n) =  -(x_n^2+x_n+1)^{3^{n-1}} \\
\times &\prod_i{\left((-1)^{n}(\rho_i^2+\rho_i+1)(\rho_i-1)^{3^{n-1}-3}x^{3^{n-1}}+S_{n-1}(x,\rho_i)\right)}\\
=&-(x_n^2+x_n+1)^{3^{n-1}} \prod_i{\left((-1)^{n}(\rho_i^3-1)(\rho_i-1)^{3^{n-1}-4}x^{3^{n-1}}+S_{n-1}(x,\rho_i)\right)}.
\end{align*}
Now,
$$\rho_i^3 -1= \frac{x_n^3-2x_n^2+4x_n}{x_n^2+x_n+1}-1 = \frac{(x_n-1)^3}{x_n^2+x_n+1}.$$
Distributing $(x_n^2+x_n+1)^{3^{n-2}}$ over each term inside the product gives $R^{(n)}(x,x_n) = (-1)^{n-1} \Pi$, where $\Pi$ is the product
\begin{equation*}
\prod_i{\left((x_n^2+x_n+1)^{3^{n-2}}(\rho_i^3-1)(\rho_i-1)^{3^{n-1}-4}x^{3^{n-1}}
+(x_n^2+x_n+1)^{3^{n-2}}S_{n-1}(x,\rho_i)\right)}.
\end{equation*}
Thus, singling out the leading term leads to
\begin{align*}
\Pi &= \prod_i{(x_n^2+x_n+1)^{3^{n-2}}(\rho_i^3-1)(\rho_i-1)^{3^{n-1}-4}x^{3^{n-1}}} + \textrm{other terms}\\
& = (x_n^2+x_n+1)^{3^{n-1}}(\rho_i^3-1)^3 (\rho_i^3-1)^{3^{n-1}-4}x^{3^n} + T\\
& = (x_n^2+x_n+1)^{3^{n-1}}\left(\frac{(x_n-1)^3}{x_n^2+x_n+1}\right)^3 \left(\frac{(x_n-1)^3}{x_n^2+x_n+1}\right)^{3^{n-1}-4}x^{3^n} + T\\
& =  (x_n^2+x_n+1)(x_n-1)^{3^n-3}x^{3^n} + T,
\end{align*}
where $T$ is sum of the remaining terms.  This gives that
$$R^{(n)}(x,x_n) =(-1)^{n-1}(x_n^2+x_n+1)(x_n-1)^{3^n-3}x^{3^n}+(-1)^{n-1}T.$$
It remains to show that $S_n(x,x_n) = (-1)^{n-1}T$ satisfies the degree assertions of the lemma for $n$.  First, the degree in $x$ of $T$ is at most equal to the degree of
$$x^{3^{n-1}}x^{3^{n-1}}x^{3^{n-1}-3} = x^{3^n-3},$$
since the first summand in the individual factors of $\Pi$ has the largest degree in $x$.  Secondly, since $x_{n-1}^{3^{n-1}}$ occurs with a nonzero coefficient in $S_{n-1}(x,x_{n-1})$ and all the other terms have degree less than $3^{n-1}$, the largest degree term in $T$ comes from the product of the terms $(x_n^2+x_n+1)^{3^{n-2}}S_{n-1}(x,\rho_i)$, whose leading term is
\begin{align*}
(x_n^2+x_n+1)^{3^{n-1}}& (\rho_1 \rho_2 \rho_3)^{3^{n-1}} = (x_n^2+x_n+1)^{3^{n-1}}(\rho_1^3)^{3^{n-1}}\\
& = (x_n^2+x_n+1)^{3^{n-1}}\left(\frac{x_n^3-2x_n^2+4x_n}{x_n^2+x_n+1}\right)^{3^{n-1}}\\
& = (x_n^3-2x_n^2+4x_n)^{3^{n-1}}.
\end{align*}
Hence, $\textrm{deg}_{x_n} S_n(x,x_n) = 3^n$.  This completes the proof.
\end{proof}

\noindent {\bf Corollary.} a) {\it The degree of $R_n(x)$ is $\textrm{deg}(R_n(x)) = 2\cdot3^n-1$.} \smallskip

\noindent b) {\it The sum of the periodic points of $\mathfrak{f}(z)$ whose periods divide $n$ is $3^n-4$, for $n \ge 1$.}

\begin{proof} Part a) follows immediately from Lemma \ref{lem:3} on setting $x_n = x$.  For part b), Lemma \ref{lem:3} shows that $\textrm{deg}_x S_n(x,x) \le 2 \cdot 3^n - 3$.  This shows that the coefficient of the term $x^{2 \cdot 3^n-2}$ in $(-1)^{n-1}R_n(x)$ comes from the leading term $(x^2+x+1)(x_n-1)^{3^n-3}x^{3^n}$.  The sum of the roots of this polynomial is
clearly $3^n-3-1 = 3^n-4$.
\end{proof}

For the discussion to follow, we denote by $\textsf{K}_3$ the maximal unramified, algebraic extension of the $3$-adic field $\mathbb{Q}_3$, and $\mathfrak{o}$ its valuation ring.

\begin{prop} For $n \ge 1$, the polynomial $R_n(x)$ has distinct roots.  Thus, the expression
$$P_n(x) = \prod_{k \mid n}{R_k(x)^{\mu(n/k)}}$$
is a polynomial with distinct roots whose degree is given by
$$\textrm{deg}(P_n(x)) = 2\sum_{k \mid n}{\mu(n/k)3^k}, \ \ n \ge 2.$$
\label{prop:14}
\end{prop}

\begin{proof}
The congruence for $R_n(x)$ in Proposition \ref{prop:13} shows that the roots of $x^{3^n}-x$, except for $x=1$, are simple roots of $R_n(x)$ (mod $3$).  Hensel's Lemma implies that $R_n(x)$ is divisible over the $3$-adic field $\mathbb{Q}_3$ by $N_n$ irreducible factors of degree $n$ in $\mathbb{Z}_3[x]$, where $\mathbb{Z}_3$ is the ring of $3$-adic integers and
$$N_n = \frac{1}{n} \sum_{d \mid n}{\mu(n/d)3^d}, \ \ n \ge 2.$$
Since $R_k(x)$ divides $R_n(x)$ if $k \mid n$, this shows that $R_n(x)$ is divisible over $\mathbb{Q}_3$ by a product of $\sum_{k \mid n}{N_k}-3$ distinct, irreducible (and nonlinear) factors whose total degree is
$$\sum_{k \mid n}{\sum_{d \mid k}{\mu(k/d)3^d}} -3= \sum_{d \mid n}{3^d \sum_{e \mid n/d}{\mu(e)}} -3 = 3^n -3.$$
Let $q(x)$ denote one of these factors.  If $\eta$ is one of its roots in $\textsf{K}_3$, with period $k$, then there are $\eta_1, \dots, \eta_{k-1} \in \textsf{K}_3$ for which
$$f(\eta, \eta_1) = f(\eta_1, \eta_2) = \cdots = f(\eta_{k-1}, \eta) = 0.$$
The congruence $f(1,y) \equiv 2(y + 2)^3$ (mod $3$) shows that if some $\eta_i \equiv 1$ (mod $3$), then $\eta \equiv 1$ (mod $3$) as well, which is false, since its degree over $\mathbb{F}_3$ is greater than 1.  Now equation (\ref{eqn:4.8}) implies, with $B(x) = \frac{x+2}{x-1}$, the chain of equations
$$f(B(\eta_1), B(\eta)) = \cdots = f(B(\eta_{k-1}),B(\eta_{k-2})) = f(B(\eta),B(\eta_{k-1})) = 0.$$
Reading backwards, this shows that $B(\eta)$ is also periodic of period $k$.  Moreover $\eta_i \not \equiv 1$ (mod $3$) in $\textsf{K}_3$ and
$$B(\eta_i)=\frac{\eta_i+2}{\eta_i-1} \equiv \frac{\eta_i+2}{\eta_i+2} \equiv 1 \ (\textrm{mod} \ 3).$$
Thus, the orbit of $B(\eta)$ consists entirely of numbers which are congruent to $1$ (mod $3$), and is therefore disjoint from the set of roots enumerated above.  Since $B(x)$ is linear fractional with determinant $-3 \neq 0$, the mapping $\eta \rightarrow B(\eta)$ is $1-1$ on $\textsf{K}_3$.  This shows that $R_n(x)$ has an additional set of $3^n-3$ distinct roots, all of which are $1$ (mod $3$).  Hence, $R_n(x)$ has at least $2(3^n-3)+5 = 2 \cdot 3^n-1$ distinct roots, when we include the factors of $R_1(x) = x(x-1)(x+2)(x^2+2)$.  But then the Corollary to Lemma \ref{lem:3} shows that these are the only roots of $R_n(x)$.  Hence, the roots of $R_n(x)$ are distinct.  This implies easily that $P_n(x)$ is a polynomial with the given degree, using that
$$R_n(x) = \pm \prod_{d \mid n}{P_d(x)}.$$
\end{proof}

\noindent {\bf Corollary.} {\it For $n \ge 2$, the average of the periodic points of primitive period $n$ is $\displaystyle \frac{1}{2}$.}

\begin{proof} For $n \ge 2$, the sum of the periodic points of primitive period $n$ is
$$\sum_{k \mid n}{\mu(n/k) (3^k-4)} = \sum_{k \mid n}{\mu(n/k) 3^k},$$
by the Corollary to Lemma \ref{lem:3}.  This is exactly half of the number of periodic points of primitive period $n$, by Proposition \ref{prop:14}.
\end{proof}

The following lemma is similar to \cite[Lemma 4.1]{mc}.

\begin{lem} There is a unique solution $y \in \textsf{K}_3$ of the equation
$$f(x,y) = 0, \ \textrm{for} \ x \in \textsf{K}_3, \ x \not \equiv 1 \ (\textrm{mod} \ 3),$$
given by the convergent series
\begin{align}
\notag y = T(x) &= \frac{1}{3}(x^3 + 2) + \frac{1}{3} (x^3 + 8) \sum_{n=0}^\infty{\binom{1/3}{n} (-1)^n \frac{3^{2n}}{(x^3+8)^n}} \\
 \label{eqn:5.1} & + \frac{1}{3} (x^3 - 1)\sum_{n=0}^\infty{\binom{-1/3}{n} (-1)^n \frac{3^{2n}}{(x^3+8)^n}}, \ \ x \not \equiv 1 \ (\textrm{mod} \ 3).
\end{align}
\label{lem:4}
\end{lem}
\begin{proof} Suppose that $t, z \in \textsf{K}_3$ are two solutions of $f(x,y) = 0$, for a given $x \not \equiv 1$ (mod $3$).  Then $f(x,t) = f(x,z) = 0$ implies that
$$0 = (z^2+z+1)f(x,t)-(t^2+t+1)f(x,z) = (z-t)((t^2 + t + 1)z^2 + (t^2 - 5t - 2)z + t^2 - 2t + 4).$$
If $z \neq t$, then $z$ is a root of the polynomial
$$q(X) = (t^2 + t + 1)X^2 + (t^2 - 5t - 2)X + t^2 - 2t + 4,$$
whose discriminant is
$$D = -3(t + 2)^2(t - 1)^2.$$
Note that $t \neq 1,-2$, since
$$f(x,1) = 3(x - 1)(x^2 + x + 1) \ \ \textrm{and} \ \ f(x,-2) = 3(x + 2)(x^2 - 2x + 4),$$
where the quadratics have no roots in $\textsf{K}_3$.  Thus $t = 1$ or $-2$ would imply $x = t$, contrary to our assumption on $x$.  It follows that $z$ lies in an extension of $\mathbb{Q}_3(t)$ whose discriminant is $-3$, and therefore in which $3$ is ramified.  This contradiction proves that $z= t$. \medskip

Now using the expression
\begin{equation*}
\mathfrak{f}(z) = \frac{1}{3} (z^3 + 8) \left(\frac{z^3-1}{z^3+8}\right)^{1/3} + \frac{1}{3}(z^3 - 1)\left(\frac{z^3-1}{z^3+8}\right)^{-1/3} + \frac{1}{3}(z^3 + 2)
\end{equation*}
and the fact that $\frac{z^3-1}{z^3+8} = 1-\frac{9}{z^3+8}$, the root $y=T(x)$ of $f(x,y) = 0$ has the $3$-adic expansion (\ref{eqn:5.1}).  To show convergence, we need to show that
\begin{equation}
3^{n+\lfloor n/2\rfloor}\binom{1/3}{n}, \ 3^{n+\lfloor n/2\rfloor}\binom{-1/3}{n} \in \mathbb{Z}_3,
\label{eqn:5.2}
\end{equation}
where $\mathbb{Z}_3$ is the ring of integers in $\mathbb{Q}_3$.  But
$$3^n \binom{1/3}{n} = (-1)^n \frac{(-1) 2 \cdot 5 \cdots (3n-4)}{n!},$$
and the power of $3$ dividing $n!$ is well-known to be
$$w_3(n!) = \frac{n - s_n}{2} \le \lfloor n/2 \rfloor,  \ \ n \ge 1,$$
where $s_n$ is the sum of the base-$3$ digits of $n$.  A similar argument for $\binom{-1/3}{n}$ proves (\ref{eqn:5.2}).  This shows that the power $3^{n-\lfloor n/2 \rfloor}$ divides the $n$-th terms of both series in (\ref{eqn:5.1}), and implies convergence, for any $x \in \textsf{K}_3$ not congruent to $1$ modulo $3$.
\end{proof}

By this proof, the terms in both infinite sums in (\ref{eqn:5.1}) are divisible by $9$, for $x \in \mathfrak{o}$ and $n \ge 4$.  This implies the congruence
\begin{align*}
T(x) & \equiv x^3 + 3 + \frac{1}{3} (x^3 + 8) \sum_{n=1}^3{\binom{1/3}{n} (-1)^n \frac{3^{2n}}{(x^3+8)^n}} \\
 & \ \ + \frac{1}{3} (x^3 - 1)\sum_{n=1}^3{\binom{-1/3}{n} (-1)^n \frac{3^{2n}}{(x^3+8)^n}} \\
 &\equiv x^3 + 3 -3\frac{2x^6 + 41x^3 + 326}{(x^3+8)^3} \\
 & \equiv x^3 \ (\textrm{mod} \ 3),
 \end{align*}
 for
 $$x \in \textsf{D} = \{x \in \mathfrak{o} \ | \ x \not \equiv 1 \ (\textrm{mod} \ 3)\}.$$
This shows that the function $T$ maps $\textsf{D}$ to itself and can be iterated on this set.  In other words, $T$ is a lift of the Frobenius automorphism on the set $\textsf{D}$. \medskip

Now let $\mathfrak{p}$ be a prime divisor of $\wp_3$ in $\Omega_{2f}$.  The field  $\Omega_{2f}$ is embedded in its completion  $\Omega_{2f, \mathfrak{p}}$, which can be viewed as a subfield of $\textsf{K}_3$, since $3$ is not ramified in $\Omega_{2f, \mathfrak{p}}/\mathbb{Q}_3$.  Recall that if $L$ is the decomposition field of $\mathfrak{p}$ in $\Omega_{2f}/K$ and $\mathfrak{p}_L$ is the prime divisor below $\mathfrak{p}$ in $L$, then
$$\textrm{Gal}(\Omega_{2f}/L) \cong \textrm{Gal}(\Omega_{2f, \mathfrak{p}}/L_{\mathfrak{p}_L}) = \textrm{Gal}(\Omega_{2f, \mathfrak{p}}/\mathbb{Q}_3).$$
(See \cite[pp. 51-55]{hak} or \cite[p. 170]{n}.)  Moreover, this group is cyclic, generated by $\tau_3$, which may be applied to elements of $\Omega_{2f, \mathfrak{p}}$.  Hence, considering the equation $f(\eta, \eta^{\tau_3}) = 0$ in $\Omega_{2f, \mathfrak{p}} \subset \textsf{K}_3$ and using Lemma \ref{lem:4} yields that
\begin{equation}
\eta^{\tau_3} = T(\eta), \ \ \eta = 2c(w/3).
\label{eqn:5.3}
\end{equation} 
For this note that $2c(w/3) \not \equiv 1$ (mod $3$) in $\textsf{K}_3$.  If this were not true, then $x = 2c(w/3)$ would satisfy
$$0 = x^3 - \beta(w/3)x+2 \equiv -\beta(w/3) \ (\textrm{mod} \ \mathfrak{p}),$$
which is not the case, since $\beta(w/3)$ is divisible by $\wp_3'$ but is relatively prime to $\wp_3$.  Applying $\tau_3$ to (\ref{eqn:5.3}) gives further that
$$\eta^{\tau_3^n} = T^n(\eta), \ \ n \ge 1.$$
Hence, the period of $\eta  = 2c(w/3)$ with respect to the $3$-adic function $T(x)$ is equal to the order of $\tau_3$ in  $\textrm{Gal}(\Omega_f/K)$.

\begin{prop}
The minimal period of $\eta  = 2c(w/3)$ with respect to the function $\mathfrak{f}(z)$ is equal to the order of the automorphism $\tau_3 = \left(\frac{\Omega_{2f}/K}{\wp_3}\right)$ in $\textrm{Gal}(\Omega_{2f}/K)$.
\label{prop:15}
\end{prop}

\begin{proof}
If $\eta_1, \dots, \eta_{n-1} \in \overline{\mathbb{Q}}$ satisfy
$$f(\eta, \eta_1) = f(\eta_1, \eta_2) = \cdots = f(\eta_{n-1},\eta) = 0,$$
then the arguments $\eta, \eta_i$ are roots of $R_n(x)$, none of which are congruent to $1$ mod $\mathfrak{p}$.  (See the proof of Proposition \ref{prop:14}.)  By Proposition \ref{prop:13} they all lie in an extension of $K$ which is unramified over $3$, since they are roots of $\frac{x^{3^n}-x}{x-1}$ mod $\mathfrak{P}$, for some prime divisor $\mathfrak{P} \mid \mathfrak{p}$.  They can therefore be considered as elements of $\textsf{K}_3$.  Now Lemma \ref{lem:4} implies that
$$\eta_1 = T(\eta), \eta_2 = T(\eta_1) = T^2(\eta), \dots, \eta_{n-1} = T(\eta_{n-2}) = T^{n-1}(\eta).$$
Hence, $\eta = T^n(\eta) = \eta^{\tau_3^n}$, which shows that $n$ is divisible by the order of $\tau_3$, since $\eta = 2c(w/3)$ generates the field $\Omega_{2f}$.  This proves the proposition.
\end{proof}

\section{Examples.}

\noindent {\bf Example 1: $d=8$.}\medskip

The minimal polynomial of $\beta(w/3)$ corresponding to $d=8$ is $p_8(x)=x^2+4x+6$.  Then
\begin{equation*}
C_8(x)=x^2\cdot p_8\left(\frac{x^3+2}{x}\right)=(x^2+2)(x^4+2x^2+4x+2).
\end{equation*}
Note that these factors appear in our Table \ref{tab:2}. Since $d$ is even and $h(-8)=1$, the polynomial $C_8(x)$ does split into a factor of degree $2h(-d)$ and a factor of degree $4h(-d)$. Upon solving for the roots of the quartic factor of $C_8(x)$ in terms of radicals and checking the approximations from Maple, we find that $w=k+\frac{\sqrt{-d}}{2}=13+\sqrt{-2}$ gives us the periodic point
$$\eta=2c(w/3)=\frac{-1-\sqrt{-1}+\sqrt{-2}}{\sqrt{2}}.$$
\medskip

\noindent {\bf Example 2: $d=11$.}\medskip

If $d=11$, then $p_{11}(x)=x^2+2x+12$ and we have
\begin{equation*}
C_{11}(x)=x^2\cdot p_{11}\left(\frac{x^3+2}{x}\right)=x^6+2x^4+4x^3+12x^2+4x+4,
\end{equation*}
which also appears in our Table \ref{tab:2}. Here $-d\equiv 5$ (mod $8$) and $h(-11)=1$, so we know that $C_{11}(x)$ is irreducible with degree $6h(-11)=6$. In order to find $c(w/3)$ in this case we solve for a root of the cubic equation
$$x^3+(1 - \sqrt{-11})x+2 = 0.$$
With $w = \frac{1}{2}(13+\sqrt{11}i)$ we have $\beta(w/3) = -1+\sqrt{11}i$, and Cardan's formulas yield that
$$2c(w/3) = \left(-1 - \frac{(4i + \sqrt{11})\sqrt{3}}{9}\right)^{1/3} + \frac{(-1 + \sqrt{11}i)}{3}\left(-1 - \frac{(4i + \sqrt{11})\sqrt{3}}{9}\right)^{-1/3};$$
where the first cube root in this formula is $0.7568 - 0.9552i$ to $4$ decimal places.  By Theorem \ref{thm:1} this value generates the ring class field $\Omega_2$ of $K = \mathbb{Q}(\sqrt{-11})$ over $\mathbb{Q}$. \medskip

\noindent {\bf Example 3: $d=56$.}\medskip

If $d=56$, the following polynomial is the minimal polynomial of $\beta(w/3)$ over $\mathbb{Q}$, where $w=7+\sqrt{-14}$:
$$p_{56}(x) = x^8 - 4x^7 + 184x^6 + 96x^5 + 532x^4 - 4224x^3 + 16416x^2 + 47520x + 156816.$$
In this case,
\begin{align*}
C_{56}(x) &= h_1(x) h_2(x) = (x^8 + 4x^7 + 8x^6 + 52x^4 + 32x^2 - 32x + 16)\\
& \times (x^{16} - 4x^{15} + 4x^{14} + 32x^{13} - 28x^{12} + 8x^{11} + 160x^{10} + 100x^8\\
& - 240x^7 + 512x^6 + 1344x^5 + 2288x^4 + 1664x^3 + 704x^2 + 16).
\end{align*}
A root of the $16$-th degree factor $h_2(x)$ generates the ring class field $\Omega_2$ of $K = \mathbb{Q}(\sqrt{-14})$ over $\mathbb{Q}$.  From Table 1, we find that the roots of this factor are periodic points of $\mathfrak{f}(z)$ with period $4$.  This can also be seen from the modular factorization
$$h_2(x) \equiv (x^4 + 2x + 2)(x^4 + x^3 + 2x^2 + 2x + 2)(x + 2)^8 \ (\textrm{mod} \ 3),$$
since this shows that $\tau_3$ has order 4 in $\textrm{Gal}(\Omega_2/K)$.  The roots of the $8$-th degree factor $h_1(x)$ also have period $4$ with respect to $\mathfrak{f}(z)$, and generate the Hilbert class field of $K$ over $\mathbb{Q}$.  One of these roots is $-1/c_1(w/6) = -1/c(w/6+3/2) = -1/c(\frac{16+\sqrt{-14}}{6})$.  The polynomial $h_1(x)$ has the property that
$$x^8 h_1(-2/x) = 2^4 h_1(x),$$
and this can be used to find radical expressions for the roots. We have
$$h_1(x) = x^4 q\left(x-\frac{2}{x}\right), \ \ q(x) = x^4+4x^3+16x^2+24x+92.$$
We first find a root $u$ of $q(x)=(x^2 + 2x + 6)^2+56$ to be
$$u = -1 - \sqrt{-5 + 2i\sqrt{14}};$$
and then solve for a root of $h_1(x)$ from $x^2-ux-2=0$:
\begin{equation*}
\frac{-1}{c(\frac{16+\sqrt{-14}}{6})} = -\frac{\sqrt{-5 + 2i\sqrt{14}}}{2} - \frac{1}{2} - \frac{1}{2} \sqrt{4 + 2i\sqrt{14} + 2\sqrt{-5 + 2i\sqrt{14}}}.
\end{equation*}
In this equation $\sqrt{-5 + 2i\sqrt{14}} = \sqrt{2}+i\sqrt{7}$.  Reciprocating yields the value
\begin{equation*}
c\left(\frac{16+\sqrt{-14}}{6}\right)=-\frac{1}{4} (\sqrt{2}+i\sqrt{7}) - \frac{1}{4} + \frac{1}{4} \sqrt{2(1+\sqrt{2})(\sqrt{2}+i\sqrt{7})}.
\end{equation*}

\noindent {\bf Example 4: $d=23, 92$.}\medskip

We first note that
$$p_{23}(x) = x^6 + 11x^5 + 65x^4 + 191x^3 + 441x^2 + 405x + 675,$$
with $\textrm{disc}(p_{23}(x)) = -3^{18} 5^4 7^4 23^3$, from \cite[p. 873]{mc}.  Then
\begin{align*}
C_{23}(x) &= (x^6 + 5x^4 - 3x^3 + 12x^2 + 4x + 8)(x^6 - x^5 + 6x^4 + 3x^3 + 10x^2 + 8)\\
& \ \times (x^6 + x^5 + x^4 + 7x^3 + 11x^2 + 5x + 1)\\
& = q_1(x) q_2(x) q_3(x),
\end{align*}
where the discriminants of $q_1(x)$ and $q_2(x)$ are divisible by $2$.  This shows that Proposition \ref{prop:18} does not generally hold for the prime $l = 2$.  Also, with
$$p_{92}(x)= x^6 - 13x^5 + 841x^4 - 2567x^3 + 1071x^2 - 20493x + 75141$$
we find that
\begin{align*}
C_{92}(x) & = q_4(x) q_5(x) = (x^6 - 10x^5 + 44x^4 - 56x^3 + 16x^2 - 32x + 64)\\
& \ \times(x^{12} + 10x^{11} + 43x^{10} + 58x^9 + 73x^8 + 328x^7 + 377x^6 - 38x^5\\
& \  - 113x^4 - 212x^3 + 207x^2 - 6x + 1).
\end{align*}
Using the group $H = \{x, A(x), B(x), U(x)\}$, with $U(x) = -2/x$, we find that the action of $H$ induces permutations on the set $\textsf{S} = \{q_1, q_2, q_3, q_4\}$.  We compute that
$$(x-1)^6 q_1(B(x)) = 3^3 q_2(x), \ \ (x-1)^6 q_k(B(x)) = 3^3 q_k(x), \  k= 3, 4.$$
Also
$$(x+2)^6 q_k(A(x)) = 6^3 q_k(x), \ k = 1, 2; \ \ (x+2)^6 q_3(A(x)) = 6^3 q_4(x).$$
It follows that the orbits of $H$ acting on $\textsf{S}$ are $\{q_1, q_2\}$ and $\{q_3, q_4\}$.  The roots of the polynomials $q_k(x)$, for $1 \le k \le 4$, have period $3$ with respect to $\mathfrak{f}(z)$, while the roots of $q_5$ have period $6$:
$$q_5(x) \equiv (x^6 + x^5 + x^4 + 1)(x + 2)^6 \ (\textrm{mod} \ 3).$$
Furthermore,
\begin{align*}
&(x+2)^{12} q_5(A(x)) = 3^6 q_6(x)\\
&  = 3^6(x^{12} + 12x^{11} + 828x^{10} + 1696x^9 - 1808x^8 + 1216x^7 + 24128x^6\\
& \  - 41984x^5 + 18688x^4 - 29696x^3 + 44032x^2 - 20480x + 4096).
\end{align*}
A root of $q_6(x)$ clearly generates the same field $\Omega_4/\mathbb{Q}$ for $K = \mathbb{Q}(\sqrt{-23})$ that a root of $q_5(x)$ does, so $q_6(x)$ must be a factor of $C_{368}(x)$.  (See the formula just before (\ref{eqn:4.8}) and the proof below.)  Moreover, $\{q_5, q_6\}$ is also an orbit under the action of the group $H$.  Thus, the action of $H$ connects factors of the polynomials $C_d(x), C_{2^2d}(x), C_{4^2 d}(x)$, for $d = 23$.  This pattern continues, as we prove now.

\begin{prop}
a) If the irreducible polynomial $q(x)$ divides $C_{d}(x)$, where $q(2c(w/3)) = 0$ and $\textrm{deg}(q(x)) = \delta$, then 
$\hat q(x) = q(-2)^{-1}(x+2)^{\delta}q(A(x)) \in \mathbb{Z}[x]$ is monic and
\begin{equation*}
\hat q(x) = q(0)^{-1}x^{\delta}q(-2/x) \mid C_{4d}(x).
\end{equation*}

 b) Assume $4 \mid d$ or $-d \equiv 1$ (mod $8$).  If $C_{4d} = \hat q(x) \tilde q(x)$, with $\textrm{deg}(\tilde q) = 2\textrm{deg}(\hat q(x)) = 2\delta$, and if $q(x)$ in a) satisfies $q(0) = 1$, then $\tilde q(0) = 1$. \smallskip
 
 c) If $-d \equiv 1$ (mod $8$) and $C_d(x) = q_1(x) q_2(x) q_3(x)$, where $q_3(x)$ has $2c(w/3)$ as a root, then $q_3(0) = 1$.  Hence $2c(w/3)$ is a unit. \smallskip
 
 d) If $-d \equiv 1$ (mod $24$) and $w_k \in \textsf{R}_{-4^k d}$ is defined as in Theorem \ref{thm:1} (with $4^k d$ in place of $d$), then $2c(w_k/3)$ is a unit in $\Omega_{2^{k+1} f}$.
 \label{prop:17}
\end{prop}

\begin{proof}
First note that $\delta$ is always even, by the Remark following Theorem \ref{thm:1}. Let $w \in \textsf{R}_{-d}$ be as in that theorem.  Proposition \ref{prop:2} implies that
$$A(2c(w/3)) = \frac{1-2c(w/3)}{1+c(w/3)} = 2c(-3/2w).$$
Propositions \ref{prop:4} and \ref{prop:7} give that
$$\frac{(2c(-3/2w))^3+2}{2c(-3/2w)} = \beta(-3/2w)$$
and
$$\beta(-3/2w) = \frac{3\beta(2w/3)+18}{\beta(2w/3)-3}.$$
Since the last value is a root of $p_{4d}(x)$ by Lemma \ref{lem:1} and its proof and the fact the roots of $p_{4d}(x)$ are permuted by the map $x \rightarrow \frac{3x+18}{x-3}$, we deduce that $A(2c(w/3)) = 2c(-3/2w)$ is a root of $C_{4d}(x)$.  Now $C_{4d}(x) = \hat q(x) \tilde q(x)$, where $\textrm{deg}(\hat q(x)) = 4h(-d)$ and $\textrm{deg}(\tilde q(x)) = 8h(-d)$.  Hence $A(2c(w/3))$ must be a root of $\hat q(x)$, since $2c(w/3)$ generates $\Omega_{2f}$ and therefore has degree $4h(-d)$ over $\mathbb{Q}$.  The fact that $q(x)$ and $\hat q(x)$ are irreducible implies that
$$(x+2)^{\delta} q(A(x)) = k \hat q(x),$$
for some constant $k$.  An elementary calculation shows that
\begin{align*}
(x+2)^{\delta} q(A(x)) & = (x+2)^{\delta} \prod_{\alpha}{(A(x)-\alpha)}\\
& =  \prod_{\alpha}{(2-2x-\alpha(2+x))}\\
& = \prod_{\alpha}{(-2-\alpha)} \prod_{\alpha}{(x-A(\alpha))} = q(-2) \hat q(x).
\end{align*}
We also know that $-2/2c(w/3) = -1/c(w/3) = \frac{c_1(w/3)}{c(2w/3)}$ is a root of $C_{4d}(x)$, by Lemma \ref{lem:1} and Proposition \ref{prop:12}.  Therefore, $U(x) = -2/x$ also takes $q(x)$ to $\hat q(x)$ and 
\begin{equation*}
x^{\delta}q_2(-2/x) = \prod_{\alpha}{(-2-\alpha x)} = \prod_{\alpha}{\alpha} \prod_{\alpha}{\left(x+\frac{2}{\alpha}\right)} = q(0) \hat q(x).
\end{equation*}
This proves a). \medskip

For part b), the last equation implies that $\hat q(0) = \frac{2^{\delta}}{q(0)}$.  Assume first that $4 \mid d$ and $\delta = 4h(-d)$.  The definition
$$C_{4d}(x) = x^{4h(-d)} p_{4d}\left(\frac{x^3+2}{x}\right) = (x^3+2)^{4h(-d)} + xr(x)$$
implies that $C_{4d}(0) = 2^{\delta} = \hat q(0) \tilde q(0)$.  If $q(0) = 1$, then $\hat q(0) = 2^{\delta}$ and this gives that $\tilde q(0) = 1$.  Thus, if the roots of $q(x)$ are units, then the roots of $\tilde q(x)$ are also units.  Secondly, if $-d \equiv 1$ (mod $8$), then $\delta = 2h(-d)$ and $\textrm{deg}(p_{4d}(x)) = 2h(-4d) = 2h(-d)$.  In this case, the formula
$$C_{4d}(x) = x^{2h(-d)} p_{4d}\left(\frac{x^3+2}{x}\right) = (x^3+2)^{2h(-d)} + xr(x)$$
again implies $C_{4d}(0) = 2^{\delta}$, and the same argument applies as before.  This proves b).  \medskip

To prove c), consider the action of $U(x) = -2/x$ on the roots of $C_d(x) =q_1(x) q_2(x) q_3(x)$.  As in the first paragraph of the proof, if $q_3(2c(w/3)) = 0$, then $x^{\delta}q_3(x)$ is a factor of $C_{4d}(x)$, and is the only factor of $C_{4d}(x)$ having degree $\delta = 2h(-d)$.  Now one of the roots of $C_d(x)$ is $x = c_1(w/6)/c(w/3) = -1/c(w/6)$, by Proposition \ref{prop:12}, which is a root of $q_1(x)$, say.  For this root we have
$$\frac{-2}{x} = \frac{2c(w/3)}{c_1(w/6)} = 2c(w/6).$$
If $w = \frac{a+\sqrt{-d}}{2}$ is chosen so that $2^3 3^2 \mid a^2 + d$ (possible because $-d$ is a square in $\mathbb{Q}_2$ and in $\mathbb{Q}_3$) then $w/6$ is the basis quotient for a proper ideal in $\textsf{R}_{-d}$, from which it follows that $\beta(w/6)$ is a root of $p_d(x)$.  (See \cite[p. 861]{mc}.)  This implies that $2c(w/6)$ is a root of $C_d(x)$, by its definition.  Similarly, $y = -1/c_1(w/6)$ must be a root of $q_2(x)$ and
$$\frac{-2}{y} = 2c_1(w/6) = 2c((w+9)/6),$$
where $(w+9)/6 = (a+18+\sqrt{-d})/12$ can also be chosen to be a basis quotient for a proper ideal in $\textsf{R}_{-d}$.  Hence, $x^{\delta}q_i(-2/x)$ must be a factor of $C_d(x)$, for $i = 1,2$.  As above, we have
\begin{align*}
x^{\delta}q_1\left(\frac{-2}{x}\right) = q_1(0) q_i(x), \ \ i = 1 \ \textrm{or} \ 2;\\
x^{\delta}q_2\left(\frac{-2}{x}\right) = q_2(0) q_{i'}(x), \ \ i' = 1 \ \textrm{or} \ 2.
\end{align*}
Multiplying these two equations and replacing $x$ by $-2/x$ gives
\begin{equation*}
\left(\frac{-2}{x}\right)^{2\delta} q_1(x) q_2(x) = q_1(0)q_2(0) q_1\left(\frac{-2}{x}\right)q_2\left(\frac{-2}{x}\right),
\end{equation*}
which implies that
$$2^{2\delta} = (q_1(0) q_2(0))^2.$$
Hence, $q_1(0) q_2(0) = 2^{\delta}$.  On the other hand,
$$C_d(0) = 2^{2h(-d)} = 2^{\delta} = q_1(0) q_2(0) q_3(0)$$
implies that $q_3(0) = 1$ and $2c(w/3)$ is a unit. \medskip

Now part d) follows from a)-c) and Theorem \ref{thm:1} by induction.
\end{proof}

\noindent {\bf Remark.} By the above argument for the factor $q(x)$ of $C_d(x)$, we also have
\begin{equation}
(x-1)^{\delta}q(B(x)) = q(1) q(x).
\label{eqn:6.7}
\end{equation}
Setting $x=0$ in $B(x) = \frac{x+2}{x-1}$ gives further that
\begin{equation}
q(-2) = q(1) q(0).
\label{eqn:6.8}
\end{equation}
It is curious that equation (\ref{eqn:6.8}) relates the values of $q(x)$ at the fixed points of the function $\mathfrak{f}(z)$. \medskip

In Example 4, with $d = 23$, the unit $2c((7+\sqrt{-23})/6)$ in the Hilbert class field $\Omega_1$ of $K = \mathbb{Q}(\sqrt{-23})$ is given by the formula
\begin{align*}
2c\left(\frac{7+\sqrt{-23}}{6}\right) &= \frac{\omega}{6}A^{1/3}+\frac{2\omega^2}{3}(2+\sqrt{-23})A^{-1/3}+\frac{-1+\sqrt{-23}}{6},\\
A  &= -76-24\sqrt{-3}+16\sqrt{-23}+12\sqrt{69},
\end{align*}
where $\omega = \frac{1}{2}(-1+\sqrt{-3})$ and $A^{1/3} = 3.30315...+1.116849...i$.

\section{Algorithm for determining $d$.}

The above discussion raises the question: if $q(x)$ is an irreducible factor of $R_n(x)$, for some $n \ge 2$, how can we determine the value of $d$ for which $q(x) \mid C_d(x)$; and consequently the discriminant $d_K$ and conductor $f$ for which $-d = d_K f^2$? \medskip

The following algorithm can be used to answer this question.  Suppose we have a given $q(x)$ which is irreducible and divides $R_n(x)$. \medskip

\noindent {\bf Step 1}. First compute the polynomial $c p(x) = \textrm{Res}_y(xy - y^3-2, q(y))$, for a monic $p(x) \in \mathbb{Z}[x]$ and some constant $c$.  By Theorem 3, $p(x) =p_d(x)$ for some discriminant $-d \equiv 1$ (mod $3$).  \medskip

\noindent {\bf Step 2}. Next, factor the polynomial $C(x) = x^\delta p\left(\frac{x^3+2}{x}\right)$, $\delta = \textrm{deg}(p(x))$.  According as $C(x)$ has $3, 2$ or $1$ irreducible factor, we know $-d \equiv 1$ (mod $8$), $4 \mid d$, or $-d \equiv 5$ (mod $8$), respectively.  Since one of the factors will be $q(x)$, this tells us that $\textrm{deg}(q(x)) = 2h(-d), 4h(-d)$, or $6h(-d)$. \medskip

\noindent {\bf Step 3}. Compute the integer factorization of $D = \textrm{disc}(q(x))$.  The odd prime factors of the integer $d$ must divide $D$, since any such prime $p$ must be ramified in the field $\Omega_f$ or $\Omega_{2f}$.  Since $\Omega_f$ and $\Omega_{2f}$ are normal over $\mathbb{Q}$, we can eliminate primes $p$ as divisors of $d$, if $q(x)$ has an irreducible factor (mod $p$) of degree $r$ and multiplicity $1$.  In this case, Hensel's Lemma and the theory of Galois extensions imply that $q(x)$ has to factor into irreducibles of degree $r$ in $\mathbb{Q}_p$, each of which generates an unramified extension of $\mathbb{Q}_p$. \medskip

\noindent {\bf Step 4.} Let $\mathcal{S}$ be the set of remaining odd prime divisors of $D$.  Form products $\tilde d$ of primes in $\mathcal{S}$, combined with powers of $4$ or $2 \cdot 4^k$, for which $h(-\tilde d)$ is the value determined in Step 2 for $\textrm{deg}(q(x))$.  These values can be eliminated if they do not satisfy $\tilde d \equiv 2$ (mod $3$) and the congruence modulo $4$ or $8$ determined in Step 2.  Moreover, if $l \neq 2,3$ is a prime divisor of $D$ eliminated in Step 3 or not dividing $\tilde d$, or even any such prime divisor of $\textrm{disc}(p(x))$, and $\left(\frac{-\tilde d}{l}\right) \neq -1$, then this product $\tilde d$ can be eliminated, by a theorem of Deuring which says that $\left(\frac{-d}{l}\right) \neq 1$ for any prime divisor $l$ of $\textrm{disc}(H_{-d}(x))$.  (This uses the fact that any prime divisor of $\textrm{disc}(p(x))$, other than $2$ or $3$, must also be a prime divisor of $\textrm{disc}(H_{-d}(x))$.) \medskip

\noindent {\bf Step 5.} If a number of such products $\tilde d$ still remain, some can be eliminated as follows.  Find primes of the form $l = x^2+y^2\tilde d$ and factor $q(x)$ mod $l$.  For all such primes, if $\tilde d = d$, $q(x)$ should factor completely (mod $l$).  If it does not, for a single prime $l$, the product $\tilde d$ can be eliminated.  This may be combined with the Legendre symbol condition in Step 4 and these two steps computed interchangeably. \medskip
 
Since $\mathcal{S}$ is a finite set, and the possible powers of the prime divisors of $\tilde d$ are limited by the powers of those primes dividing $D$, this algorithm must come to an end after finitely many steps.  However, it may happen that more than one product $\tilde d$ remains as a candidate for the correct value of $d$.  In this case $q(x)$ can be factored over $K = \mathbb{Q}(\sqrt{-\tilde d})$.  If it does not factor, then $\tilde d$ can be eliminated as a possibility.  The condition on the class number $h(-d)$ may also be used to clinch the final value of $d$.  \medskip

In some cases, computing the effect of the transformations in the group $H = \{x,A(x),B(x),U(x)\}$ can help narrow down the search for $d$.  For example, let $q(x)$ be the last polynomial of degree $16$ mentioned after Table 1.  Computing $p(x)$ and $C(x) = q(x) \hat q(x)$ shows that $q(x)$ is the factor of smaller degree of $C(x)$, so $\textrm{deg}(q(x)) = 2h(-d)$ and therefore $h(-d) = 8$.  Furthermore, $q(x)$ maps to itself under all the elements of $H$, so $q(x)$ is not the result of applying $A(x)$ to a polynomial dividing $C_{d'}(x)$, for a smaller value of $d'$, as in Proposition \ref{prop:17}a).  Since $d$ must be even, we conclude that $d$ is either $4$ or $8$ times a product of odd primes.  Now, $\textrm{disc}(q(x)) = 2^{184} 3^{56} 5^8 7^{16} 11^8 23^8$, but
\begin{align*}
q(x) & \equiv  (x^2 + 2x + 4)(x^2 + 4x + 2)(x + 1)^2(x^2 + 4x + 1)(x^2 + 2)^2(x + 3)^2\\
& \ \times (x^2 + x + 2) \ \textrm{mod} \ 5;\\
q(x) & \equiv (x^2 + 9x + 12)(x^2 + 11x + 2)(x + 21)^2(x + 1)^2(x + 12)^2(x^2 + 10x + 8)\\
& \ \times (x + 19)^2(x^2 + 12x + 2) \ \textrm{mod} \ 23.
\end{align*}
Thus, $\mathcal{S} =\{7, 11\}$.  The possible values for $d$ lie in $\{8 \cdot 7^r, 8 \cdot 11^s, 4 \cdot 7^r \cdot 11^s, 8 \cdot 7^r \cdot 11^s: r, s \ge 1\}$.  This is because the values $4 \cdot 7^r \equiv 1$ (mod $3$) can be eliminated; while $-11 \equiv 5$ (mod $8$), so $3$ divides the class number for $-\tilde d  = -4 \cdot 11^s$, because $s$ must be odd.  (See Example 2 in Section 6.)  The value $\tilde d = 56$ is not possible by Example 3, and $8 \cdot 7^2$ is not a possibility because $l = 3^2+2\cdot 7^2 = 107$ splits in the corresponding ring class field (note $\left(\frac{-8}{107}\right) = 1$), but $q(x)$ factors into a product of quadratic polynomials mod $l$.  The product $\tilde d = 8 \cdot 7^2$ can also be eliminated because $\left(\frac{-8}{19}\right) = 1$ and $19$ divides $\textrm{disc}(p(x))$.  The values $\tilde d = 8 \cdot 7^r, 8 \cdot 11^s$ and $\tilde d = 8 \cdot 7^r \cdot 11^s$ are similarly not possible because $\left(\frac{-8}{19}\right) = 1$ and $7$ and $11$ are quadratic residues of $19$. \medskip

In the remaining products, $\tilde d = 4 \cdot 7^r \cdot 11^s$, $11$ must occur to an odd power for $\tilde d$ to be $2$ mod $3$.  By the formulas in Section 4.4, this power must be $1$, since otherwise $11$ would divide the class number.  For $h(-\tilde d) = 8$, the only possibility that remains is $d = 4\cdot 77 = 308$.  Hence, $q(x)$ is a factor of $C_{308}(x)$, as claimed in Section 4.2.  Note that $\left(\frac{-77}{p}\right) = -1$ for $p = 5, 23$, as well as for all the prime divisors $p\neq 2, 3, 7, 11$ of
$$\textrm{disc}(p(x)) = 2^{184}3^{168}5^{48}7^{32}11^{16}19^823^{16}29^883^2191^4227^6281^4.$$
This is a strong check on the arguments.  In fact, $\left(\frac{-77}{p}\right) = -1$ also holds for all $27$ prime divisors, distinct from $2, 7$ or $11$, of the class equation $H(x)$ corresponding to the polynomial $p(x)$ (see \cite{mc}).  Hence, we have $H(x) = H_{-308}(x)$.

\noindent Dept. of Mathematics, University of Missouri at Columbia \smallskip

\noindent 208 Math Sci Bldg, 810 Rollins St., Columbia, MO, 65211 \smallskip

\noindent email: akkarapakams@missouri.edu \bigskip

\noindent Dept. of Mathematical Sciences \smallskip

\noindent Indiana University at Indianapolis (IUI) \smallskip

\noindent LD 270, 402 N. Blackford St., Indianapolis, IN, 46202 \smallskip

\noindent e-mail: pmorton@iu.edu

\end{document}